\documentclass[oneside,english]{amsart}
\usepackage[T1]{fontenc}
\usepackage[latin9]{inputenc}
\usepackage{geometry}
\usepackage{amstext}
\usepackage{amsthm}
\usepackage{amssymb}
\usepackage{graphicx}
\usepackage{esint}
\usepackage{comment}
\usepackage{ulem}

\makeatletter
\numberwithin{equation}{section}
\numberwithin{figure}{section}
\theoremstyle{plain}
\newtheorem{thm}{\protect\theoremname}
\theoremstyle{plain}
\newtheorem{prop}[thm]{\protect\propositionname}
\theoremstyle{plain}
\newtheorem{lem}[thm]{\protect\lemmaname}
\newtheorem{cor}[thm]{\protect\corollaryname}
\theoremstyle{remark}
\newtheorem{rem}[thm]{\protect\remarkname}
\theoremstyle{remark}
\newtheorem{ex}[thm]{\protect\examplename}
\newtheorem{conjecture}[thm]{\protect\conjecturename}

\usepackage{todonotes}
\global\long\def\Re{\operatorname{Re}}

\global\long\def\Im{\operatorname{Im}}

\global\long\def\Arg{\operatorname{Arg}}

\global\long\def\Log{\operatorname{Log}}

\global\long\def\Res{\operatorname{Res}}

\makeatother

\usepackage{babel}
\providecommand{\lemmaname}{Lemma}
\providecommand{\corollaryname}{Corollary}
\providecommand{\propositionname}{Proposition}
\providecommand{\remarkname}{Remark}
\providecommand{\theoremname}{Theorem}
\providecommand{\examplename}{Example}
\providecommand{\conjecturename}{Conjecture}
\providecommand{\problemname}{Problem}

\usepackage{epsfig}
\usepackage[usenames,dvipsnames]{pstricks}
\usepackage{pst-grad}
\usepackage{pst-plot}

\def\eref#1{(\ref{#1})}

\def\sinh{{\rm sinh}}

\begin{document}
\title{Sheffer sequences with zeros on a line}
\author{G.-S. Cheon${}^1$}
\address{${}^1$Department of Mathematics/ Applied Algebra and Optimization Research Center, Sungkyunkwan University, Suwon 16419, Rep. of Korea}
\email{gscheon@skku.edu}
\thanks{ G.-S. Cheon was partially supported by the National Research Foundation of Korea (NRF) grant funded by the Korean government (MSIT) (RS-2025-00573047 and 2019R1A2C1007518).}
\author{T. Forg\'acs${}^{\dag, 2}$}

\address{${}^2$Department of Mathematics, California State University, Fresno, Fresno, CA 93740-8001, USA}
\email{tforgacs@mail.fresnostate.edu}
\email{khangt@mail.fresnostate.edu}
\author{K. Tran${}^2$}

\maketitle
\begin{abstract}  We extend a result of Bump et al. to show that a large family of Sheffer sequences has their zeros - up to perhaps a finite number of exceptions - on a vertical line. We connect a particular such sequence to the Riemann zeta function via a product representation of a scaled Mellin transform, analogously to the product decomposition of a Mellin transform involving the generalized Laguerre polynomials into factors of the Gamma function and Meixner polynomials. \\


\end{abstract}

\section{Introduction}
In recent works (\cite{cft}, \cite{cfmt} and \cite{cfkt}) the authors studied the zero locus of various Sheffer sequences, and investigated combinatorial properties of the coefficients of such polynomial sequences. In the most recent of these works we have shown that the zeros of a certain family of Sheffer sequences, as well as those of their cognate sequences lie on a line. The current work adds to the understanding of polynomial sequences with zeros on a line. In particular, we extend a result of Bump et al. (\cite{BBump1}) which shows that the zeros of the sequence of polynomials generated by 
\[
(1-z)^{s-1-\alpha/2}(1+z)^{-2-\alpha/2}
\]
 lie on the line $\Re s=1/2$. We show (c.f. Theorem \ref{thm:pn_generalization}) that for any choice of $p, p^*$, the zeros of the polynomials generated by 
 \[
 (1-\alpha_{0}z)^{p+s}(1+\alpha_{0}z)^{p^{*}-s}\prod_{i=1}^{N}(1-\alpha_{i}^{2}z^{2})^{p_{i}}
 \]
  lie on the line $\Re s = \frac{p^*-p}{2}$. In additon to this extension theorem, we also establish a connection between a sequence of polynomials like those addressed Theorem \ref{thm:pn_generalization}, a scaled Mellin transform of a family of functions related to the generealized Laguerre polynomials, and the Riemann zeta function (see Theorem \ref{thm:mellinzeta}). Through this connection, understanding the zeros of the scaled Mellin transform becomes equivalent of understanding the zeros of the zeta function.\\
 The paper is organized as follows: in Section \ref{sec:2} we review the setup and result of Bump, and state the generalization of that result. We also explore some properties of a sequence of polynomials arising from our work, and develop an equation which connects a scaled Mellin transform to a product involving the zeta function, along with our polynomials. Section \ref{sec:proofmainThm} is dedicated in its entirety to the proof of Theorem \ref{thm:pn_generalization}, and the Appendix contains some elementary - albeit lengthy - supplemental calculations we rely on in the proof of Theorem \ref{thm:pn_generalization}.
 \section{The polynomials with only zeros on a vertical line}\label{sec:2}

Let $f(x)$ be a complex-valued function on $(0,\infty)$. The Mellin transform of $f(x)$ is an analytic function of the complex 
variable $s$ defined by
$$
\int_0^\infty f(x)x^{s-1}dx. 
$$

Bump et al. \cite{BBump,BBump1} studied the properties of the zeros of the Mellin transforms of the Laguerre functions defined by
$$
{\frak L}_n^{(\alpha)}(x)=x^{\alpha/2}e^{-x/2}L_n^{(\alpha)}(x),\quad \alpha>-1, n=0,1,2,\ldots
$$
where $L_n^{(\alpha)}(x)$ are generalized Laguerre polynomials \cite{Andrews} generated by 
\begin{eqnarray}\label{e:Laguerre}
\sum_{n=0}^\infty L_n^{(\alpha)}(x)z^n={1\over (1-z)^{\alpha+1}}e^{-zx/(1-z)}.
\end{eqnarray}
The Mellin transforms of ${\frak L}_n^{(\alpha)}(x)$ are computed by 
\begin{eqnarray}\label{e:Mellin}
M_n^{(\alpha)}(s):=\int_0^\infty {\frak L}_n^{(\alpha)}(x)x^{s-1}dx =2^{s+\alpha/2}\Gamma(s+\alpha/2)P_n^{(\alpha)}(s), \quad n=0,1,2,\ldots
\end{eqnarray}
where $\Gamma$ is the gamma function and the polynomials $P_n^{(\alpha)}(s)$ are generated by the relation
\begin{eqnarray}\label{e:BB}
\sum_{n=0}^\infty P_n^{(\alpha)}(s)z^n =(1-z)^{s-1-\alpha/2}(1+z)^{-s-\alpha/2}.
\end{eqnarray} 
The zeros of the polynomials $P_n^{(\alpha)}(s)$ lie on the critical line ${\rm Re}(s)={1\over2}$ (see \cite{BBump1}). Due to this property, a connection between these polynomials and the Riemann hypothesis has been suggested. In addtition, the Mellin transforms of these polynomials are of much interest in connection with generalizations of Riemann's second proof of the analytic continuation of the Riemann zeta function $\zeta(s)$. B. Bump's result on the location of the zeros of the $P_n^{(\alpha)}(s)$ is a special case of the following theorem, which we will prove in Section \ref{sec:proofmainThm}.

\begin{thm}
\label{thm:pn_generalization} 
Let $\alpha_i$, $i=0,1,2, \ldots, N$, $p_i$, $i=1, 2, \ldots, N$, $p$ and $p^*$ be real numbers. Suppose that $|\alpha_0| <|\alpha_1|<\cdots<|\alpha_N|$. Define the sequence
$\{h_{n}(s)\}$ by the generating relation
\begin{eqnarray}\label{e:main}
\sum_{n=0}^{\infty}h_{n}(s)\frac{z^{n}}{n!}=(1-\alpha_{0}z)^{p+s}(1+\alpha_{0}z)^{p^{*}-s}\prod_{i=1}^{N}(1-\alpha_{i}^{2}z^{2})^{p_{i}}.
\end{eqnarray}
If $p^{*}+p\le0$, then for all $n$, the zeros of $h_{n}(s)$ lie on the vertical
line ${\rm Re}(s)=(p^{*}-p)/2=:c$. If $p^{*}+p>0$, then the same conclusion
holds except for $2\lceil c+p\rceil$ zeros. 
\end{thm}
\begin{rem} The polynomials $P_n^{(\alpha)}(s)$ as in \eqref{e:BB} are obtained from Theorem \ref{thm:pn_generalization} by setting $p=-1-\frac{\alpha}{2}, p^*=-\frac{\alpha}{2}$, $\alpha_0=1$ and $N=0$. Then $(p^*-p)/2=1/2$, and the the right hand side of \eqref{e:main} reproduces the right hand side of \eqref{e:BB}.
\end{rem}
\subsection{A sequence of ``intermediate'' polynomials} In this section we find an explicit formula for the polynomials $h_n(s)$ in terms of another sequence $\{ q_n(s)\}$. From equation \eqref{e:main} it is clear that we may view the polynomials $h_n(s)$ as generalizations of the polynomials $P_n^{(\alpha)}$ (c.f. equation \eqref{e:BB}). We are interested in the ``intermediate'' generalizations $q_n(s)$ not only because their combinations constitute the poynomials $h_n(s)$, but also because they have some simple properties in common with the Riemann zeta function which we wish to explore. \\
Before we emark on the discussion of the sequence $\{ q_n(s)\}$, we briefly review some basic Riordan theory. A polynomial sequence $\{p_n(s)\}_{n=0}^{\infty}$ is said to be a \textit{Sheffer sequence} for the pair $(g, f)$ if there exist $g, f \in \mathbb{C}[[z]]$ with $g(0)\ne0$, $f(0)=0$, and $f'(0)\ne0$, such that
\begin{equation} \label{eq:sheffer}
    g(z)e^{sf(z)} = \sum_{n=0}^{\infty} p_n(s) \frac{z^n}{n!},
\end{equation}
where $p_n(s)$ is a polynomial of degree $n$. Riordan matrices \cite{Lou} are fundamental tools in the study of Sheffer sequences in terms of generating functions. It is known that $\{p_n(s)\}_{n\ge0}$ is a Sheffer sequence for $(g,f)$ if and only if its coefficient matrix $A = [p_{n,k}]_{n,k\ge0}$ is an \emph{exponential Riordan matrix} given by
$$
    p_{n,k} = \frac{n!}{k!} [z^n]\, g(z)f(z)^k,
$$
where $[z^n]$ denotes the coefficient extraction operator. As is customary, we denote this matrix as $A = \langle g(z), f(z) \rangle$ or $A = \langle g, f \rangle$. If $b = (b_0, b_1, \dots)^T$ is a column vector generated by $b(z) = \sum_{n\ge0} b_n z^n/n!$, then the action of the matrix $A = \langle g(z), f(z) \rangle$ on $b(z)$ is given by
$$
    \langle g(z), f(z) \rangle b(z) = g(z) b(f(z)),
$$
known as the \emph{Riordan fundamental property}. Given two matrices $\langle g, f \rangle$ and $\langle u, v \rangle$, their product is given in terms of generating functions as follows:
$$
\langle g, f \rangle \cdot \langle u, v \rangle = \langle g \cdot u(f), v(f) \rangle.
$$
This multiplication is known as the \textit{Riordan product}.\\
With these basic notions in hand, we now discuss the polynomials $\{ q_n(s)\}$ and their relation to Theorem \ref{thm:pn_generalization}. Using the notation of theorem, let 
$$
G_0(z)=(1-\alpha_0z)^p(1+\alpha_0z)^{p*}\quad{\text and}\quad G_i(z)=(1-\alpha_i^2z^2)^{p_i}\quad {\text for}\quad i\ge1,
$$
and consider the generating functions
$$
g(z) = \prod_{i=0}^{N} G_i(z), \qquad \textrm{and} \qquad f(z) = \ln\left( \frac{1 - \alpha_0 z}{1 + \alpha_0 z} \right).
$$
Then the sequence $\{h_n(s)\}_{n=0}^\infty$ can be regarded as the Sheffer sequence for the pair $(g, f)$.
Using the Riordan product, the matrix $\langle g(z), f(z) \rangle$ can be factored as
$\left\langle \prod_{i=1}^{N} G_i(z), z \right\rangle \left\langle G_0(z), f(z) \right\rangle.$
Consequently, applying the Riordan fundamental property yields
\begin{align}
    \sum_{n=0}^{\infty} h_n(s) \frac{z^n}{n!} &= g(z)e^{sf(z)} = \langle g(z), f(z) \rangle e^{sz} \notag \\
    &= \left\langle \prod_{i=1}^{N} G_i(z), z \right\rangle \left( G_0(z) e^{sf(z)} \right). \label{eq:factorized-sheffer}
\end{align}
We expand the product $\prod_{i=1}^{N} G_i(z)$ as a power series:
$$
    \prod_{i=1}^{N} G_i(z) = \sum_{k=0}^{\infty} b_k z^k,
$$
and let $\{q_n(s)\}_{n=0}^{\infty}$ be the Sheffer sequence for the pair $(G_0, f)$. Then the $(n,k)$-entry of the matrix $\left\langle \prod_{i=1}^{N} G_i(z), z \right\rangle$ is $\frac{n!}{k!} b_{n-k}$.
Moreover, we have
\begin{equation} \label{e:Lagu}
    G_0(z)e^{sf(z)} = \sum_{n=0}^{\infty} q_n(s) \frac{z^n}{n!} = (1 - \alpha_0 z)^{p + s}(1 + \alpha_0 z)^{p^* - s}.
\end{equation}
From equation \eqref{eq:factorized-sheffer} we hence deduce that
\begin{equation} \label{eq:hns-convolution}
    h_n(s) = \sum_{k=0}^{n} \frac{n!}{k!} b_{n-k} q_k(s).
\end{equation}
Once more, setting $\alpha_0=1$, $p=-1-\alpha/2$ and $p^*=-\alpha/2$ in equation \eqref{e:Lagu}, we immediately obtain the right hand side of equation \eqref{e:BB} so that the polynomials $q_n(s)$ are a direct generalization of the polynomials $P_n^{(\alpha)}(s)$. The next result is an immediate consequence of Theorem \ref{thm:pn_generalization}. 

\begin{cor}\label{thm:pn_generalization-1} Let $\{q_n(s)\}_{n=0}^\infty$ be the polynomial sequence generated by \eref{e:Lagu}.
If $p^{*}+p\le0$, then all the zeros of $q_{n}(s)$ lie on the vertical
line ${\rm Re}(s)=(p^{*}-p)/2=:c$. If $p^{*}+p>0$, then the same conclusion
holds except for $2\lceil c+p\rceil$ zeros. 
\end{cor}

Substituting $z$ by $z/\alpha_{0}$ \footnote{The $\alpha_0=0$ case is degenerate, as in this case all polynomials $\{h_n(s)\}$ are constant.} allows us henceforth to assume that $\alpha_{0}=1$ in \eref{e:Lagu}. We note that the polynomials $q_n(s)$ have basic properties in common with the Riemann zeta function $\zeta(s)$. For example, it can be easily shown that the $q_n(s)$ satisfy the functional equation
\begin{eqnarray*}
q_n(s)=(-1)^n q_n(p^*-p-s),
\end{eqnarray*}
which is a standard result in the theory of the zeta function. In particular, if $p^*-p=1$ then $q_n(s)=(-1)^nq_n(1-s)$. Using the generating relation in equation \eqref{e:Lagu}, Newton's Binomial Theorem and the binomial convolution for exponential generating functions we obtain the following explicit formula for the $q_n(s)$:
$$
q_n(s)=n!\sum_{k=0}^n(-1)^k{p+s\choose k}{p^*-s\choose n-k}, \qquad n=0,1,2,3,\ldots
$$
For the readers' convenience, we list the first few polynomials $q_n(s)$ for $p=0, p^*=1$ in Table \ref{table:q_ns}.
\begin{table}[htp]
\caption{The $q_n(s)$ for small $n$, $p=0, p^*=1$}
\begin{center}
\begin{tabular}{cc}
$n$ & $q_n(s)$ \\ 
\hline
\hline
$0$ & $1$ \\
$1$ & $-2s+1$ \\
$2$ & $4s^2-4s$ \\
$3$ & $-8s^3+12s^2-4s$\\
$4$ & $16s^4-32s^3+32s^2-16s$\\
\hline
\end{tabular}
\end{center}
\label{table:q_ns}
\end{table}%

\begin{thm} For $n \geq 0$, the polynomials $q_n(s)$ given in equation \eqref{e:Lagu} satisfy the following recurrence relation:
\begin{eqnarray}\label{e:recur}
2q_{n+1}(s)=(p^*-p-2s)q_n(s)+(p^*-s)q_n(s+1)-(p+s)q_n(s-1),
\end{eqnarray}
with the intial condition $q_{0}(s)=1$.
\end{thm}
\begin{proof} In order to simplify notation, let 
\begin{eqnarray}\label{e:GF}
 \varphi(z,s):=(1-z)^{p+s}(1+z)^{p^{*}-s}=\sum_{n=0}^{\infty}q_{n}(s){z^n/n!}.
\end{eqnarray}
By simple computation we obtain
\begin{eqnarray}\label{e:recur-1}
\sum_{n=0}^{\infty}q_{n+1}(s){z^n\over n!}&=&{\partial\over \partial z}\sum_{n=0}^{\infty}q_{n}(s){z^n\over n!}={\partial\over \partial z}(1-z)^{p+s}(1+z)^{p^{*}-s}\nonumber\\
&=&(1-z)^{p+s-1}(1+z)^{p^*-s-1}(p^*-p-2s-p^*z-pz).
\end{eqnarray}
Since 
$$
p^*-p-2s-p^*z-pz={1\over2}\left((p^*-p-2s)(1-z)(1+z)+(p^*-s)(1-z)^2+(-p-s)(1+z)^2\right),
$$
it follows from equation \eref{e:recur-1} that
$$
\sum_{n\ge0}q_{n+1}(s){z^n\over n!}={1\over2}\left((p^*-p-2s)\varphi(z,s)+(p^*-s)\varphi(z,s+1)+(-p-s)\varphi(z,s-1)\right),
$$
which proves the claim.
\end{proof}
The following theorem - while interesting in its own right - will be used to establish an interlacing property of the zeros of the polynomials $q_n(s)$ (c.f Theorem \ref{interlace})
\begin{thm}
For $n\ge0$, the polynomials $q_n$ given in equation \eqref{e:Lagu} satisfy the following three-term recurrence relation:
\begin{eqnarray}  \label{e:precur2}
q_{n+1}(s)=\left(p^*-p-2s\right)q_n(s)+n\left(n-p^*-p-1\right)q_{n-1}(s).
\end{eqnarray}
with the initital conditions  $q_0(s)=1$ and $q_{-1}(s)=0$.
\end{thm}

\begin{proof} Differentiating equation \eqref{e:GF} successively we obtain 
\begin{equation*}
\sum_{n=0}^{\infty} nq_{n-1}(s){\frac{z^n}{n!}}=z\varphi(z,s)\quad{\text{and}}%
\quad \sum_{n=0}^{\infty} n^2q_{n-1}(s){\frac{z^n}{n!}}=z\varphi(z,s)+z^2{\frac{%
\partial }{\partial z}}\varphi(z,s).  
\end{equation*}
Multiplying the right hand side of (\ref{e:precur2}) by ${\frac{z^n}{n!}}$%
, and taking the summation over all $n\ge0$ yields 
\begin{eqnarray*}
&&(p^*-p-2s)\sum_{n=0}^{\infty}q_{n}(s){\frac{z^n}{n!}}%
+\sum_{n=0}^{\infty} n^2q_{n-1}(s){\frac{z^n}{n!}}%
-(p+p^*+1)\sum_{n=0}^{\infty} nq_{n-1}(s){\frac{z^n}{n!}} \\
&&=(p^*-p-2s)\varphi(z,s)+z\varphi(z,s)+z^2{\frac{%
\partial }{\partial z}}\varphi(z,s)-(p+p^*+1)z\varphi(z,s) \\
&&=(1-z)^{p+s-1}(1+z)^{p^*-s-1}\left(p^*-p-2s-pz-p^*z)\right) \\
&&=(1-z)^{p+s-1}(1+z)^{p^*-s-1}\left(-(p+s)(1+z)+(q-s)(1-z)\right) \\
&&={\frac{\partial }{\partial z}}\varphi(z,s)=\sum_{n=0}^{\infty} q_{n+1}(s){%
\frac{z^n}{n!}},
\end{eqnarray*}
proving the claimed recurrence relation for all $n\ge0$.
\end{proof}

We now investigate whether, and in what sense, the zeros of the polynomials
$q_n(s)$ interlace\footnote{The interlacing property - and more restrictively the notion of being in proper position - of families of polynomials plays a central role in stability and control theory, the theory of orthogonal polynomials, and combinatorics and total positivity. Further discussion of this topic goes beyond the scope of the current paper.}. Recall (see for example Liu and Wang \cite{BLiu}) that if $f,g \in \mathbb{R}[x]$ have only real zeros, and $\{r_i\}$ and $\{s_j\}$ denote the zeros of
$f$ and $g$ in nonincreasing order respectively, then we say that $g$
{\it interlaces} $f$, denoted by $g\preceq f$, if ${\rm deg}(f)={\rm
deg}(g)+1=n$ and
\begin{eqnarray*}
r_n\le s_{n-1}\le\cdots\le s_2\le r_2\le s_1\le r_1.
\end{eqnarray*}

Using the notation of Corollary \ref{thm:pn_generalization-1} we extend this notion to the polynomials $\{ q_n(s) \}$, by saying that the zeros $\left\{c+r_{n,k} i\mid r_k\in \mathbb{R}, k \in \mathbb{N}\right\}$ of $q_n(s)$ are in nonincreasing order if $\{r_{n,k}\}$
is in nonincreasing order. The concept of the interlacing property is extended to the family $\{ q_n(s)\}$ analogously.
\begin{lem}\label{Liu}\cite{BLiu}
Let $F,g,f\in\mathbb{R}[x]$ satisfy the following conditions:
\begin{itemize}
\item[{\rm (i)}] $F(x)=a(x)f(x)+b(x)g(x)$, where $a,b\in\mathbb{R}[x]$ such that ${\rm deg}(F)={\rm deg}(f)$ or ${\rm
deg}(f)+1$.
\item[{\rm (ii)}] All zeros of $f$ and $g$ are real and $g\preceq
f$.
\item[{\rm (iii)}] $F$ and $g$ have leading coefficients of the same
sign.
\end{itemize}
Suppose that $b(r)\le0$ whenever $f(r)=0$. Then all zeros of $F$ are
real and $f\preceq F$.
\end{lem}

\begin{thm}\label{interlace} Consider the setup of Theorem \ref{thm:pn_generalization} along with the ensuing discussion. If $p^*+p\le0$ then the zeros of the polynomials
$q_n(s)$ interlace, i.e.
\begin{eqnarray*}
q_n(s)\preceq q_{n+1}(s),\;\; n\ge1.
\end{eqnarray*}
\end{thm}
\begin{proof} Let $p^*+p\le0$ and $\widehat q_n(s)=i^nq_n(c+is)$ with $c={p^*-p\over 2}$. It can be readily shown that
$\widehat q_n(s)\in\mathbb{R}[s]$ and all zeros of $\widehat q_n(s)$ are real. Thus it
suffices to show that $\widehat q_n(s)\preceq \widehat q_{n+1}(s)$ for $n\ge1$. A direct
computation shows that $\widehat q_1(s)\preceq \widehat q_2(s)$. Let $n\ge2$. Substituting
$c+is$ into $s$ in \eref{e:precur2}, we obtain
\begin{eqnarray*}
\widehat q_{n+1}(s)=2s \widehat q_n(s)-n(n-p^*-p-1)\widehat q_{n-1}(s).
\end{eqnarray*}
Set
$$
F= \widehat q_{n+1}(s),\;f= \widehat q_n(s),\;g=\widehat q_{n-1}(s)
$$
and
$$
a=2s,\;b=-n(n-p^*-p-1).
$$
Clearly, the conditions (i) and (iii) in Lemma \ref{Liu} hold for all $n\ge2$. We complete the proof by verifying that condition (ii) holds for all $n\ge2$. Since $\widehat q_1(s)\preceq \widehat q_2(s)$, the condition (ii) also
holds for $n=2$. Moreover, by the assumption $p^*+p\le0$ we have $b\le0$ for all $n\ge2$. Lemma \ref{Liu} now implies that $\widehat q_2(s)\preceq \widehat q_3(s)$, and by repeating this argument for $n=3,4,5 \ldots$ we establish that $\widehat q_n(s)\preceq \widehat q_{n+1}(s)$ for all $n\ge2$, which completes the proof.
\end{proof}

\begin{ex} Taking $n=9$, we obtain 
\begin{eqnarray*}
q_9(s)&=&362880 - 1297152 s + 1884672 s^2 - 1838080 s^3 + 919296 s^4 - 
 462336 s^5 \\
 &+& 96768 s^6 - 30720 s^7 + 2304 s^8 - 512 s^9 \\
q_{10}(s)&=&3628800 - 12971520 s + 21441024 s^2 - 18380800 s^3 + 12869120 s^4 - 
 4623360 s^5 \\
 &+& 1892352 s^6 - 307200 s^7 + 84480 s^8 - 5120 s^9 + 1024 s^{10}.
\end{eqnarray*}
Figure \ref{fig:interlaceEx} shows the location of the zeros of these polynomials.
\begin{figure}[htbp]
\begin{center}
\includegraphics[height=3 in]{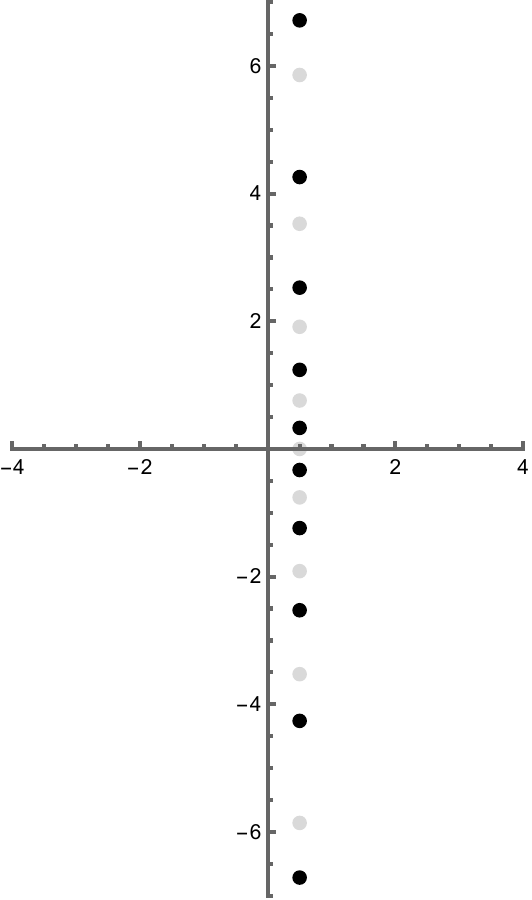}
\caption{The zeros of $q_9(s)$ and $q_{10}(s)$, with $p=-1$ and $p^*=0$}
\label{fig:interlaceEx}
\end{center}
\end{figure}

\end{ex}

\subsection{A connection to the Riemann zeta function $\zeta(s)$} 
This section establishes a connection between the Riemann zeta function, our sequence of polynomials $\{q_n(s)\}$ and a scaled Mellin transforms of a family of fucntions reltated to the generalized Laguerre polynomials.\\
We first recall  (see \cite{BAtaki}) that the Mellin transform of a family of functions related to the generalized Laguerre polynlmials $L_k^{(\alpha)}(x)$ is expressed in terms of the Meixner polynomials $\mathfrak{M}_k$: 
\begin{eqnarray}  \label{e:LMellin}
\int_0^\infty k!L_k^{(b-1)}\left((1-{c^{-1}})x\right)e^{-x}x^{s-1}dx={\Gamma}%
(s)\mathfrak{M}_k\left(-s;b,c\right),
\end{eqnarray}
where $\mathfrak{M}_k(x;b,c)$ are generated by 
\begin{eqnarray}  \label{e:Meix}
\sum_{n=0}^{\infty}\mathfrak{M}_n(x;b,c){\frac{z^n}{n!}}=\left(1-{\frac{z}{c}}%
\right)^x\left({\frac{1}{1-z}}\right)^{x+b},\quad c\ne1, \quad b \in \mathbb{R}.
\end{eqnarray}
 It is also known (see for example \cite{BPoly}) that if $\mathcal{M}(s)$ is the Mellin transform of a function $f:(0,\infty)\rightarrow\mathbb{C}$ then 
\begin{eqnarray}  \label{e:Mell2}
\int_0^\infty x^\lambda f(ax^b)x^{s-1}dx={\frac{1}{b}}a^{-{\frac{s+\lambda}{b%
}}}\mathcal{M}\left({\frac{s+\lambda}{b}}\right),\quad a>0,\;b\ne0.
\end{eqnarray}

Using the Riordan fundamental property, equation \eref{e:GF} can be written in terms of the  exponential Riordan matrix:
\begin{eqnarray}\label{e:q_n}
\sum_{n=0}^{\infty}q_{n}(s){z^n\over n!}&=&(1-z)^{p+s}(1+z)^{p^{*}-s}\nonumber\\
&=&\left<(1+z)^{p^*},z\right>(1+z)^{-s}(1-z)^{p+s}\\
&\stackrel{\eqref{e:Meix}}{=}&\left<(1+z)^{p^*},z\right> \sum_{n=0}^{\infty}\mathfrak{M}_n(-s;-p,-1){\frac{z^n}{n!}}.\nonumber
\end{eqnarray}
If we denote the Appell sequence for the pair $((1+z)^{p^*},z)$ by $\{a_n(s)\}$ and write $a_n(s)=\sum_{k=0}^n \Delta_{n,k}^{*}s^k$, then 
$$
\Delta_{n,k}^{*}={n\choose k}(p^*)_{n-k},
$$ where $(x)_n:=x(x-1)\cdots(x-n+1)$ is the falling factorial of degree $n$ with $(0)_0=1$. Consequently, by equation \eref{e:q_n}, the polynomials $q_n(s)$ can be expressed
in terms of the Meixner polynomials as follows: 
\begin{eqnarray}  \label{e:pform}
q_n(s)=\sum_{k=0}^n\Delta_{n,k}^{*}\;\mathfrak{M}_k\left(-s;-p,-1\right), \qquad n \geq 0.
\end{eqnarray}

Now, consider the {\it umbral composition} (see \cite{BRoman}) of the Appel sequence $\{a_n(s)\}$ with the Laguerre polynomial sequence $L(s)=\{n!L_n^{(-p-1)}(s)\}$ defined by  
\begin{eqnarray}\label{e:composition}
a_n(L(s))=\sum_{k=0}^n \Delta_{n,k}^{*}\;k!L_k^{(-p-1)}(s).
\end{eqnarray}
Since $\{L(s)\}$ is the Sheffer sequence for the pair $((1-z)^{p},{z\over z-1})$, from equation \eref{e:Laguerre} we see that $a_n(L(s))$ is the Sheffer sequence for the pair $((1+z)^{p^*}(1-z)^{p},{z\over z-1})$. For the remaining of this section, in order to shorten notation, we let
$$
\Phi_n(s):=a_n(L(s))e^{-s/2}=\sum_{k=0}^n {n\choose k}(p^*)_{n-k}\;k!L_k^{(-p-1)}(s)e^{-s/2}.
$$

\begin{thm}
\label{fMellin} Let $\{q_n(s)\}$ be as equation \eqref{e:Lagu}, $\Phi_n(s)$ be as above, and let $\Gamma$ be the Gamma function. Then for $n \geq 0$,
\begin{eqnarray}  \label{e:PhiMellin}
\int_0^\infty \Phi_n(2\pi x)x^{s-1}dx=q_n(s) \pi^{-s}\Gamma(s).
\end{eqnarray}
\end{thm}
\begin{proof} For $k \geq 0$, let 
\begin{eqnarray*}
f_k(x)=k!L_k^{\left(-p-1\right)}(2x)e^{-x}.
\end{eqnarray*}
Then for $n \geq 0$, we have $\Phi_n(2\pi x)=\sum_{k=0}^n \Delta^*_{n,k}\;f_k(\pi x)$.
By equation (\ref{e:LMellin}), the Mellin transform of $f_k(x)$ is
given by 
\begin{eqnarray*}
\mathcal{M}(s):=\int_0^\infty f_k(x)x^{s-1}dx= {\Gamma}(s)\mathfrak{M}_k\left(-s;-p,-1\right).
\end{eqnarray*}
Taking $\lambda=0$, $a=\pi$ and $b=1$ in (\ref{e:Mell2}) yields 
\begin{eqnarray}  \label{e:eqn}
\int_0^\infty f_k(\pi x)x^{s-1}dx=\pi^{-s}\mathcal{M}(s) =\pi^{-s}{\Gamma}(s)\mathfrak{M}_k\left(-s;-p,-1\right).
\end{eqnarray}
Multiplying both sides of (\ref{e:eqn}) by $\Delta^*_{n,k}$ and summing over $0 \leq k \leq n$ we obtain  
\begin{eqnarray*}
\int_0^\infty \left(\sum_{k=0}^n \Delta^*_{n,k}f_k(\pi x)\right)
x^{s-1}dx=\pi^{-s}\Gamma(s)\sum_{k=0}^n \Delta^*_{n,k}\mathfrak{M}_k\left(-s;-p,-1\right).
\end{eqnarray*}
Using equation \eref{e:pform} we arrive at
$$
\int_0^\infty \Phi_n(2\pi x)x^{s-1}dx=q_n(s)\pi^{-s}\Gamma(s),
$$
as claimed.
\end{proof}

\begin{rem} If we set $\lambda=0$, $a=\pi$ and $b=2$ in (\ref{e:Mell2}), we obtain the Mellin transform of $\Phi_n(2\pi x^2)$ as follows:
\begin{eqnarray}\label{e:x2}
\int_0^\infty \Phi_n(2\pi x^2)x^{s-1}dx=\frac{1}{2}q_n({s/2})\pi^{-{s/2}}\Gamma({s/2}).
\end{eqnarray}
\end{rem}
The next (and final) theorem of the section connects the polynomials $\{q_{n}(s/2)\}$ to the Riemann zeta function $\zeta(s)$ and scaled Mellin transforms of the functions $\Phi_n(2\pi x^2)$, $n \geq 0$. In order to set up the statement of the result, let $\Phi_n^*(x):=\Phi_n(2\pi x^2)$, and define $\psi_j^*(x)$ by
\begin{eqnarray}  \label{e:psiphi-1}
\psi_j^*(x)=\sum_{n=1}^\infty \Phi_j^{*}(n\sqrt{x}).
\end{eqnarray}

\begin{thm}\label{thm:mellinzeta} Let $\psi_j^*(x)$ be as in equation \eqref{e:psiphi-1}, $\{q_n(s)\}$ be as in equation \eqref{e:Lagu} and $\zeta(s)$ be the Riemann zeta function. Then for $j \geq 0$, the following equality holds:
\begin{eqnarray}  \label{e:left}
\int_0^\infty\psi^*_j(x)x^{\frac{s}{2}-1}dx=q_j({s/2})\pi^{-{s/2}}\Gamma({s/2})\zeta(s).
\end{eqnarray}
\end{thm}
\begin{proof} 
 Substituting $y=n\sqrt{x}$ and using equation \eref{e:psiphi-1} we obtain 
\begin{eqnarray*}
\int_0^\infty \psi_j^{*}(x)x^{{\frac{s}{2}}-1}dx
&=&\int_0^\infty\sum_{n=1}^\infty\Phi_j^{*}(y) \left({\frac{y}{n}}\right)^{s-1}{\frac{2}{n}}dy \\
&=&2\left(\sum_{n=1}^\infty{\frac{1}{n^s}}\right)\int_0^\infty\Phi_j^*(y)y^{s-1}dy.
\end{eqnarray*}
Since $\Phi_j^{*}(y)=\Phi_n(2\pi y^2)$, equation \eref{e:x2} implies that 
\begin{eqnarray*}  \label{e:left}
\int_0^\infty\psi_j^{*}(x)x^{{\frac{s}{2}}-1}dx=q_n({s/2})\pi^{-{s/2}}\Gamma({s/2})\zeta(s),
\end{eqnarray*}
as required.
\end{proof}

\section{The Proof of Theorem \ref{thm:pn_generalization}} \label{sec:proofmainThm}

We now turn our attention to the proof of Theorem \ref{thm:pn_generalization}. The overall structure of the proof is as follows. We start with an integral representation (on a small circle around the origin) of the polynomials of interest, and demonstrate that this integral can be computed by integrating over a Hankel contour instead of the circle. We further deform this countour into one over which the integrand has analytically desirable properties (c.f Lemma \ref{lem:zycurve}), and use the saddle point method to develop asymptotic expressions for the integral $p_n(t)$ over various ranges of $t$: (i) for when $e^{-\ln^4 n} /n \ll t =\mathcal{O}(\ln^4 n/n)$ (c.f. Lemma \ref{lem:asympsmallt}) , and (ii) for when $\ln^4 n/n \ll t <1$ (c.f. Lemma \ref{lem:globalasymp}). Having obtained estimates that are uniform in $t$, we demonstrate that the the integral over a small, `center' section of the contour dominates the integral on the tails of the contour. Finally, we  employ (and refer to) methods we have utilized in \cite{cft} to establish the claimed location of the zeros of our polynomial sequence (c.f. Lemmas \ref{lem:changeargglobal} and \ref{lem:changeargsmallt}).
\newline
\indent Consider the notation and setup of Theorem \ref{thm:pn_generalization}. The conclusions of the theorem are trivially true when $\alpha_0=0$. Henceforth we consider $\alpha_0 \neq 0$ -  in fact, using the substitution $z$ by $z/\alpha_{0}$ (which has no effect on the location of the zeros of the generated sequence), it suffices to consider the case
$\alpha_{0}=1$. From the Cauchy differentiation
formula we obtain
\begin{align*}
h_{n}(c-int) & =\frac{1}{2\pi i}\oint_{|z|=\epsilon}\frac{(1-z)^{p+c-int}(1+z)^{p+c+int}\prod_{i=1}^{N}(1-\alpha_{i}^{2}z^{2})^{p_{i}}}{z^{n+1}}dz\\
 & =\frac{1}{2\pi i}\oint_{|z|=\epsilon}\psi(z,t)e^{n\phi(z,t)}dz,
\end{align*}
where 
\begin{align}
\phi(z,t) & =it\Log(1+z)-it\Log(1-z)-\Log z,\label{eq:phidef}\\
\psi(z) & =z^{-1}\prod_{k=0}^{N}(1-\alpha_{i}z^{2})^{p_{i}}, \qquad \textrm{and}\label{eq:psidef}\\
p_{0} & =p+c.\nonumber  
\end{align}
We note that for large $n$, 
\[
\lim_{R\rightarrow\infty}\int_{C_{R}}\psi(z,t)e^{n\phi(z,t)}dz=0,
\]
where $C_{R}$ is any circular arc radius $R$. Thus 
\begin{equation}
h_{n}(c-int)=\frac{1}{2\pi i}\oint_{\Gamma_{1}\cup\Gamma_{2}}\psi(z,t)e^{n\phi(z,t)}dz\label{eq:contourdeformed}
\end{equation}
where $\Gamma_{1}$ and $\Gamma_{2}$ are two loops around the cuts
$(-\infty,-1]$ and $[1,\infty)$ (see Figure \ref{fig:contour}). 

\begin{figure}
\begin{centering}
\includegraphics[scale=0.3]{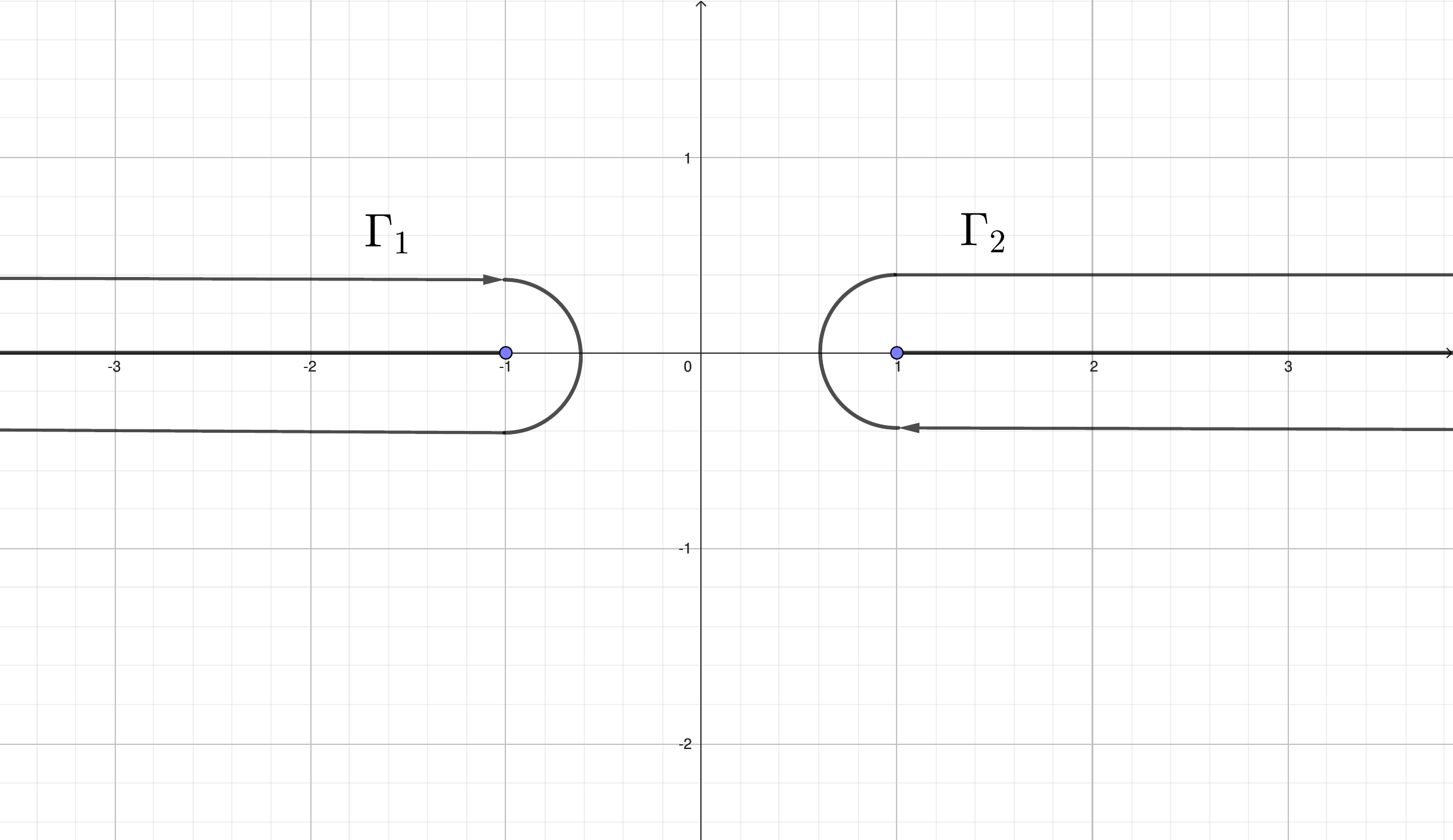}
\par\end{centering}
\caption{\label{fig:contour}$\Gamma_{1}$ and $\Gamma_{2}$ curves}

\end{figure}

Since 
\[
(1-z)^{p+c}(1+z)^{p+c}\prod_{k=1}^{N}(1-\alpha_{i}^{2}z^{2})^{p_{i}}
\]
is an even function, we see using the substitution $z\rightarrow-z$
that 
\[
\int_{\Gamma_{1}}\psi(z,t)e^{n\phi(z,t)}dz(-1)^{n}=\int_{\Gamma_{2}}\psi(z)e^{n\phi(z,-t)}dz.
\]
On the other hand, the substitution $z\rightarrow\overline{z}$ yields
\[
\int_{\Gamma_{2}}\psi(z)e^{n\phi(z,-t)}dz=\overline{(-1)^{n+1}\int_{\Gamma_{2}}\psi(z)e^{n\phi(z,t)}dz}.
\]
We deduce from (\ref{eq:contourdeformed}) that $\pi h_{n}(c-int)$
is either the imaginary part, or $-i$ times the real part of 
\begin{equation} \label{eq:pn}
p_{n}(t)=\int_{\Gamma_{2}}\psi(z)e^{n\phi(z,t)}dz,
\end{equation}
depending of whether $n$ is even or odd.

We now commence the analysis of the asymptotic behavior of $p_{n}(t)$ as $n\rightarrow\infty$ using the saddle point method. Observe that the two solutions in $z$ of 
\begin{equation} \label{eq:phiz}
\phi_{z}(z,t)=-\frac{1}{z}+\frac{it}{1-z}+\frac{it}{1+z}=0
\end{equation}
are 
\[
\zeta_{\pm}=-it\pm\sqrt{1-t^{2}.}
\]
Clearly, for each $t\in(0,1)$, $\zeta:=\zeta_{+}$ lies on the
unit circle in the fourth quadrant. 
\begin{lem}
\label{lem:phi2zeta} Let $\phi(z,t)$ be defined as in equation \eqref{eq:phidef}. Then for each $t\in(0,1)$, 
\[
\Re\phi_{z^{2}}(\zeta(t),t)>0.
\]
Furthermore, $\Im\phi_{z^{2}}(\zeta(t),t)\le0$ if $t\in(0,1/\sqrt{2}]$
and $\Im\phi_{z^{2}}(\zeta(t),t)>0$ for $t\in(1/\sqrt{2},1)$. 
\end{lem}

\begin{proof}
Given the definition of $\phi(z,t)$, one readily calculates
\[
\phi_{z^{2}}(\zeta,t)=\frac{it}{\left(-\sqrt{1-t^{2}}+it+1\right)^{2}}-\frac{it}{\left(\sqrt{1-t^{2}}-it+1\right)^{2}}+\frac{1}{\left(\sqrt{1-t^{2}}-it\right)^{2}},
\]
along with
\begin{align*}
\Re\phi_{z^{2}}(\zeta,t) & =\frac{2(1-t)t^{4}(t+1)}{\left(\sqrt{1-t^{2}}-1\right)^{2}\left(\sqrt{1-t^{2}}+1\right)^{2}}, \qquad \textrm{and} \\
\Im\phi_{z^{2}}(\zeta,t) & =\frac{t^{3}\sqrt{1-t^{2}}\left(2t^{2}-1\right)}{\left(\sqrt{1-t^{2}}-1\right)^{2}\left(\sqrt{1-t^{2}}+1\right)^{2}}.
\end{align*}
The result is now immediate.
\end{proof}
We now want to deform the contour of integration $\Gamma_2$ into a contour through $\zeta$. An example of such a contour is given graphically in \cite{paris} for example, and Figure \ref{fig:deformed} illustrates a particluar instance of the curve we use (for the value of $t=0.3$). In the following theorem we rigorously justify the existence of such a curve. 
\begin{lem}
\label{lem:zycurve}For each $t\in(0,1)$, there exist $L>0$, and
a function $z(y)$ analytic in a neighborhood of $(-\infty,L)$ such
that 
\begin{itemize}
\item $z(0)=\zeta$, 
\item $z(L)\in(0,1)$, 
\item $z(y)$, $y\in(-\infty,L),$ lies in the fourth quadrant, 
\item and 
\begin{equation}
-y^{2}=\phi\left(z(y),t\right)-\phi\left(\zeta,t\right),\qquad\forall y\in(-\infty,L).\label{eq:zyfunc}
\end{equation}
\end{itemize}
\end{lem}
\begin{figure}[htbp]
\begin{center}
\includegraphics[width= 2.5in]{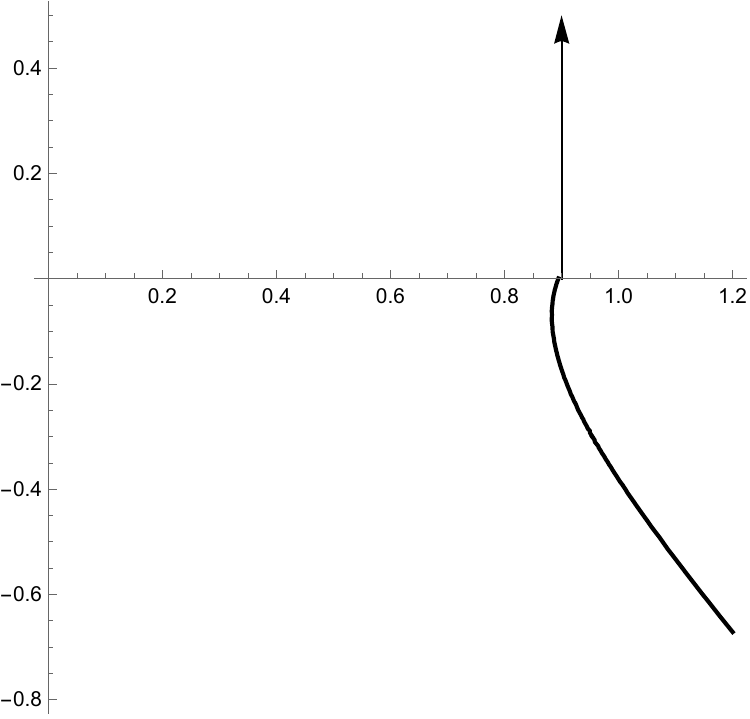}
\caption{The new contour of integration: the $z(y)$ curve along with a vertcial line replacing $\Gamma_2$}
\label{fig:deformed}
\end{center}
\end{figure}

\begin{proof}
By the definition of $\zeta$, expanding $\phi(z,t)$ about $z=\zeta$ yields
\[
\phi\left(z,t\right)-\phi\left(\zeta,t\right)=\frac{\phi_{z^{2}}(\zeta,t)}{2}(z-\zeta)^{2}(1+b(z)),
\]
where $b(z)$ is analytic in an open neighborhood (depending on $t$)
of $\zeta$ and $|b(z)|<1$. Lemma \ref{lem:phi2zeta} implies that $\phi_{z^2}(\zeta,t)$ is zero free on $t \in (0,1)$, as is $1+b(z)$. Thus they have analytic square roots (using the principal cut) and we may invert the relation 
\begin{equation}
y=\frac{\sqrt{\phi_{z^{2}}(\zeta,t)}}{\sqrt{2}i}(z-\zeta)\sqrt{1+b(z)}\label{eq:localcurve}
\end{equation}
to obtain $z(y)$ analytic in an small open neighborhood of $0$ such
that $z(0)=\zeta$ and \eqref{eq:zyfunc} holds for all $y$ in some
neighborhood of $0$. The equation above implies that 
\begin{equation}
z=\zeta+y\frac{\sqrt{2}i}{\sqrt{\phi_{z^{2}}(\zeta,t)}}+o(y)\label{eq:zsmally}.
\end{equation}
Given that $|\zeta|=1$, we see that for small $y$ 
\begin{equation}
|z|^2=1+2y\left(\Re\zeta\Re\left(\frac{\sqrt{2}i}{\sqrt{\phi_{z^{2}}(\zeta,t)}}\right)+\Im\zeta\Im\left(\frac{\sqrt{2}i}{\sqrt{\phi_{z^{2}}(\zeta,t)}}\right)\right)+o(y).\label{eq:zymod}
\end{equation}
Lemma \ref{lem:phi2zeta} implies that 
\[
\Arg\left(\frac{\sqrt{2}i}{\sqrt{\phi_{z^{2}}(\zeta,t)}}\right)\in\begin{cases}
(\pi/4,\pi/2) & \text{ if }t\in(1/\sqrt{2},1)\\{}
[\pi/2,3\pi/4) & \text{ if }t\in(0,1/\sqrt{2}]
\end{cases}.
\]
Consequently, if $t\in(0,1/\sqrt{2})$ then 
\begin{eqnarray*}
\Re \left( \frac{\sqrt{2}i}{\sqrt{\phi_{z^{2}}(\zeta,t)}}\right)&<&0, \qquad \textrm{and}\\
\Im \left( \frac{\sqrt{2}i}{\sqrt{\phi_{z^{2}}(\zeta,t)}}\right)&>&0.
\end{eqnarray*}
In addition, if $t\in(1/\sqrt{2},1)$, then
\begin{equation}
0< \Re \left( \frac{\sqrt{2}i}{\sqrt{\phi_{z^{2}}(\zeta,t)}}\right)<\Im \left( \frac{\sqrt{2}i}{\sqrt{\phi_{z^{2}}(\zeta,t)}}\right), \qquad \textrm{and} \qquad 0< \Re \zeta <-\Im \zeta. \nonumber
\end{equation}
Therefore, for all $t \in (0,1)$, $|z|<1$ for $0<y \ll1$ and $|z|>1$ for negative $y$ with $0< |y|\ll 1$.

We first extend the curve $z(y)$ inside the unit circle. To this
end, let $L$ be the supremum of the set of all positive real numbers
$l\in\mathbb{R}^{+}$ such that $z(y)$ has an analytic continuation
to an open neighborhood of $[0,l)$ and $z([0,l))$ is a subset of
the fourth quadrant. Let $y_{k}\in[0,L)$ be a sequence of real numbers
such that $y_{k}\rightarrow L$ and that $z(y_{k})$ is convergent.
Write $\lim_{k\to\infty}z(y_{k})=z(L)$. It follows from the definition
of $L$ that, $z(L)$ lies on the boundary of the fourth quadrant.
We note that for each fixed $t\in(0,1)$, 
\[
\frac{d\Im\phi(ix,t)}{dx}=\Im\left(-\frac{t}{1+ix}-\frac{t}{1-ix}-\frac{1}{x}\right)=0,\qquad(x\in\mathbb{R}\backslash\{0\}).
\]
Consequently, if $x\in(-\infty,0)$, then 
\[
\Im\phi(ix,t)=\lim_{x\rightarrow0^{-}}\Im\phi(ix,t)=\frac{\pi}{2}.
\]
On the other hand, 
\begin{align*}
\frac{d\Im\phi(\zeta,t)}{dt} & =\Im\left(\phi_{z}(\zeta,t)\frac{d\zeta}{dt}+\phi_{t}(\zeta,t)\right)=\Im\phi_{t}(\zeta,t)\\
 & =\ln\frac{|1+\zeta|}{|1-\zeta|}>0,
\end{align*}
from which we conclude that $\Im\phi(\zeta,t)$ is an increasing function
of $t$ on the interval $(0,1)$. For all $0\leq y\leq L$, \eqref{eq:zyfunc}
implies that for a fixed $t\in(0,1)$ 
\begin{eqnarray*}
\Im\phi(z(y),t) & = & \Im\phi(\zeta,t)\\
 & < & \Im\phi(-i,1)\\
 & = & \ln|1-i|-\ln|1+i|-\left(-\frac{\pi}{2}\right)\\
 & = & \frac{\pi}{2}.
\end{eqnarray*}
It follows that $z(y)$ does not lie on the negative imaginary axis.
We also claim that $z(y)$ is not on the unit circle for any $y\in(0,L]$.
For if it were, we would have a $t'\neq t$ such that $z(y)=\zeta(t')$,
necessitating that $\Im\phi(z(y),t)=\Im\phi(\zeta(t),t)=\Im\phi(\zeta(t'),t)$. Write $\zeta(t)=e^{i\theta}$ and $\zeta(t')=e^{i \theta '}$ for some $\theta, \theta' \in (-\pi/2,0)$. 
Applying the Mean Value Theroem to $\Im \phi(e^{i \theta, t})$ as a function of $\theta$, we conclude that there exists a $(-\pi/2,0) \ni \theta^* \neq \theta, \theta'$ such that
\[
\Im \left(\phi_z(e^{i \theta^*},t) i e^{i \theta^*} \right)=-\csc (\theta^*)(t+\sin \theta^*)=0.
\]
Since $\csc \theta^* \neq 0$ on $(-\pi/2,0)$ we conclude that $t+\sin \theta^*=0$, or equivalently, $\sin \theta=\sin \theta^*$, a contradiction. Since $z(y)$ is continuous, we conclude that $z(L)\in(0,1)$.

Next we extend the curve $z(y)$ outside the unit circle. Set $K$
be the infimum of $l\in\mathbb{R}^{-}$ such that $z(y)$ has an analytic
continuation to an open neighborhood of $(l,0]$ and $z((l,0])$ is
a subset of the fourth quadrant and we define $z(K)$ correspondingly.
Arguments identical to those given above also show that $z(K)\notin-i\mathbb{R}^{+}$,
and that $|z(K)|>1$. Since $z(K)$ must lie in the boundary of the
fourth quadrant, either $z(K)\in(1,+\infty)$, or $z(K)=\infty$. 

Suppose $z(K)\in(1,\infty)$. Note that if \eqref{eq:zyfunc} holds
with real $z>1$ and if $t\in(0,1)$, then 
\begin{eqnarray*}
0 & = & \Im(\phi(z,t)-\phi(\zeta,t))\\
 & = & \Im\left\{ (it\ln(1+z)-it\ln(z-1)+i(-\pi)-\ln z\right.\\
 & - & \left.it\Log(1+\sqrt{1-t^{2}}-it)+it\Log(1-\sqrt{1-t^{2}}+it)+\Log(-it+\sqrt{1-t^{2}})\right\} \\
 & = & t\ln\left(\frac{1+z}{z-1}\sqrt{\frac{1-\sqrt{1-t^{2}}}{1+\sqrt{1-t^{2}}}}\right)+\arcsin(-t).
\end{eqnarray*}
Since 
\[
\frac{d}{dz}\left(\frac{1+z}{z-1}\right)=-\frac{2}{(z-1)^{2}}<0,
\]
\[
\lim_{z\to1^{+}}\frac{1+z}{z-1}=+\infty,\qquad \textrm{and} \qquad \lim_{z\to+\infty}\frac{1+z}{z-1}=1,
\]
we see that $\ln\left(\frac{1+z}{z-1}\sqrt{\frac{1-\sqrt{1-t^{2}}}{1+\sqrt{1-t^{2}}}}\right)$
is strictly monotone with range 
\[
R=\left(\ln\sqrt{\frac{1-\sqrt{1-t^{2}}}{1+\sqrt{1-t^{2}}}}\ ,\ +\infty\right).
\]
Given that $t\in(0,1)$, the left end point of $R$ is negative, and
hence there is exactly one value of $z\in(1,\infty)$ such that 
\begin{equation}
\ln\left(\frac{1+z}{z-1}\sqrt{\frac{1-\sqrt{1-t^{2}}}{1+\sqrt{1-t^{2}}}}\right)=-\frac{\arcsin(-t)}{t},\label{eq:z>1}
\end{equation}
or equivalently, 
\[
\Im(\phi(z,t)-\phi(\zeta,t))=0.
\]
Solving \eqref{eq:z>1} for $z$ we obtain
\[
z^{*}=\frac{{\displaystyle {-\sqrt{\frac{1-\sqrt{1-t^{2}}}{1+\sqrt{1-t^{2}}}}+\exp\left[-\frac{\arcsin(-t)}{t}\right]}}}{{\displaystyle {\sqrt{\frac{1-\sqrt{1-t^{2}}}{1+\sqrt{1-t^{2}}}}-\exp\left[-\frac{\arcsin(-t)}{t}\right]}}}.
\]
In order to show that $z(K)\neq z^{*}$, it suffices to show that $\Re(\phi(z^{*},t)-\phi(\zeta,t))>0$, as this inequality shows via equation \eqref{eq:zyfunc} that $z^*$ cannot lie on the $z(y)$ curve. 
We calculate 
\begin{eqnarray*}
\Re(\phi(z^{*},t)-\phi(\zeta,t)) & = & -t\left(\Arg(1+z^{*})-\Arg(1-z^{*})\right)-\ln(z^{*})+t\left(\Arg(1+\zeta)-\Arg(1-\zeta)\right)+\ln|\zeta|\\
 & = & -t(0-\pi)-\ln(z^{*})-t\frac{\pi}{2}=t\frac{\pi}{2}-\ln(z^{*}),
\end{eqnarray*}
where we used the equality 
\[
\Arg(1+\zeta)-\Arg(1-\zeta)=-\frac{\pi}{2},
\]
which is a straightforward consequence of the Theorem of Thales and
the fact that $|\zeta|=1$. One readily verifies (perhaps with a CAS) that $t \frac{\pi}{2}> \ln z^*$ for all $t \in (0,1)$, from which the claim $\Re(\phi(z^{*},t)-\phi(\zeta,t))>0$ follows. We now
see that $z(K)\notin(1,\infty)$, and consequently it must be the
case that $z(K)=\infty$. This fact, equation \eqref{eq:zyfunc}, and the definition of $\phi(z,t)$ imply that $K=-\infty$.  The proof of the lemma is now complete. 
\end{proof}

\subsection{Asymptotics}
In this section we develop asymptotic equivalences in order to help us count the number of zeros of $p_n(t)$ via the argument principle. We accomplish this by looking at intervals $\left(\delta e^{-\ln^4 n}/n, C \ln ^4n/n \right)$ and $\left(C \ln ^4 n/n ,1 \right)$ whose union approximates $(0,1)$ as $n \to \infty$. 
\subsubsection{The case $e^{-\ln^{4}n}/n\ll t=\mathcal{O}(\ln^{4}n/n)$}
\begin{lem}
\label{lem:asympsmallt} Let $p_n(t)$ be as in \eqref{eq:pn}. Then asymptotically, as
$n\rightarrow\infty$ 
\[
p_{n}(t)\sim\frac{i\sin\pi(c+p+1-int)\Gamma(c+p+1-int)2^{c+p+1+int}\prod_{k=1}^{N}(1-\alpha_{i}^{2})^{p_{i}}}{n^{c+p+1-int}}
\]
uniformly on $e^{-\ln^{4}n}/n \ll t=\mathcal{O}(\ln^{4}n/n)$. 
\end{lem}

We rewrite the integral over $\Gamma_{2}$ as 
\[
\int_{+\infty}^{(1^{-})}\psi(z)e^{n\phi(z,t)}dz
\]
where the path of integration is the Hankel contour (see Figure \ref{fig:hankels}) looping around
the ray $[1,\infty)$. The notation $1^{-}$ means the path goes around
$1$ in the negative (clockwise) direction. We make the substitution $z \to e^{z}$
to arrive at the expression 
\[
\int_{+\infty}^{(0^{-})}\psi(e^{z})e^{n\phi(e^{z},t)}e^{z}dz.
\]
\begin{figure}[htbp]
\begin{center}
\includegraphics[width =3 in]{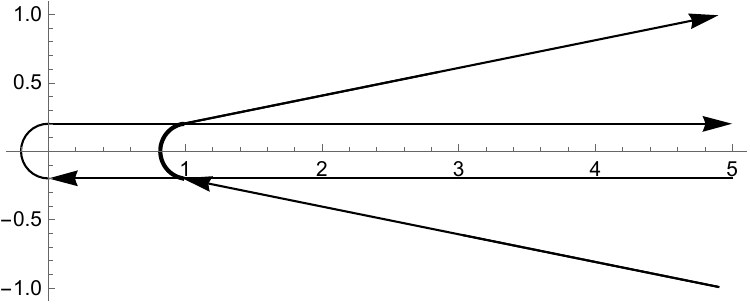}
\caption{The Hankel contours before and after the change of variable $z \to e^z$}
\label{fig:hankels}
\end{center}
\end{figure}

Suppose $\epsilon>0$ is such that $n\epsilon=o(1)$. We break up
the range of integration of the last integral as 
\begin{equation}
 \int_{(0^{-})}^{ln^{5}n/n\pm i\epsilon}\psi(e^{z})e^{n\phi(e^{z},t)}e^{z}dz + \int_{\ln^{5}n/n+i\epsilon}^{+\infty+i\epsilon}\psi(e^{z})e^{n\phi(e^{z},t)}e^{z}dz + \int_{+\infty-i\epsilon}^{\ln^{5}n/n-i\epsilon}\psi(e^{z})e^{n\phi(e^{z},t)}e^{z}dz.\label{eq:splitintegral}
\end{equation}
Using the identity 
\begin{equation}
\psi(z)e^{n\phi(z)}=\frac{(1-z)^{p+c-int}(1+z)^{p+c+int}\prod_{k=1}^{N}(1-\alpha_{i}^{2}z^{2})^{p_{i}}}{z^{n+1}},\label{eq:integrand}
\end{equation}
we conclude that the first integral in \eqref{eq:splitintegral} is
asymptotic to 
\begin{align}
 & 2^{c+p+int}\prod_{k=1}^{N}(1-\alpha_{i}^{2})^{p_{i}}\int_{(0^{-})}^{\ln^{5}n/n\pm i\epsilon}\frac{(1-e^{z})^{c+p-int}}{e^{nz}}dz\nonumber \\
\sim & 2^{c+p+int}\prod_{k=1}^{N}(1-\alpha_{i}^{2})^{p_{i}}\int_{(0^{-})}^{\ln^{5}n/n\pm i\epsilon}(-z)^{c+p-int}e^{-nz}dz.\nonumber \\
= & \frac{2^{c+p+int}\prod_{k=1}^{N}(1-\alpha_{i}^{2})^{p_{i}}}{n^{c+p+1-int}}\int_{(0^{-})}^{\ln^{5}n\pm in\epsilon}(-z)^{c+p-int}e^{-z}dz.\label{eq:mainint}
\end{align}
By Lemma 10 in \cite{cft}, if $n\epsilon=o(1)$, then
\[
\int_{(0^{-})}^{\ln^{5}n\pm in\epsilon}(-z)^{c+p-int}e^{-z}dz\sim2i\sin\pi(c+p+1-int)\Gamma(c+p+1-int)
\]
as $n\rightarrow\infty$, uniformly on $e^{-\ln^{4}n}/n\ll t=\mathcal{O}(\ln^{4}n/n)$. Thus we have established the claimed asymptotic expression for the first integral. \\ 
We now find a bound for the second term of \eqref{eq:splitintegral} given by
\[
\int_{\ln^{5}n/n+i\epsilon}^{+\infty+i\epsilon}\psi(e^{z})e^{n\phi(e^{z},t)}e^{z}dz.
\]
 For $u \in [\ln^{5}n/n, \infty)$ write $z=u+i\epsilon$. Then
\[
|1-e^{z}|=|-u-i\epsilon+\mathcal{O}(u^{2}+\epsilon^{2})| \sim u >\frac{\ln^{5}n}{n},
\]
and 
\begin{align*}
|1-\alpha_{i}^{2}e^{2z}| & >\text{|\ensuremath{\Im(1-\alpha_{i}^{2}e^{2z})|}}\\
 & =\alpha_{i}^{2}e^{2u}\sin2\epsilon\\
 & \gg\epsilon.
\end{align*}
Recall from (\ref{eq:psidef}) that
\[
\psi(z)=z^{-1}\prod_{k=0}^{N}(1-\alpha_{i}z^{2})^{p_{i}}.
\]
Consequently, 
\[
e^{z}\psi(e^{z})=\begin{cases}
\mathcal{O}\left(\frac{n^{|p+c|}}{\ln^{5|p+c|}n}+\frac{1}{\epsilon^{\sum_{i=1}^{N}|p_{i}|}}\right) & \text{ if }z=\mathcal{O}(1)\\
\mathcal{O}\left(e^{2u\sum_{i=0}^{N}|p_{i}|}\right) & \text{ if }z\gg1.
\end{cases}
\]
We also note that 
\[
\left|\left(\frac{1+e^{z}}{1-e^{z}}\right)^{int}\right|=\exp(-nt(\Arg(1+e^{z})-\Arg(1-e^{z}))<e^{\pi nt},
\]
since $\Arg(1+e^{z})-\Arg(1-e^{z})$ is an interior angle of the triangle with vertices $0$, $1+e^z$ and $1-e^{z}$.

We break the range of integration of 
\[
\int_{\ln^{5}n/n+i\epsilon}^{+\infty+i\epsilon}\psi(e^{z})e^{n\phi(e^{z},t)}e^{z}dz
\]
into two pieces: (i) $\ln^{5}n/n<|z|=\mathcal{O}(1)$ and (ii) $|z|\gg1$. The integral
over the first piece is 
\[
\mathcal{O}\left(e^{n\pi t}\left(\frac{n^{|p+c|}}{\ln^{5|p+c|}n}+\frac{1}{\epsilon^{\sum_{i=1}^{N}|p_{i}|}}\right)\frac{e^{-\ln^{5}n}}{n}\right),
\]
while  the integral over the second piece is 
\[
\mathcal{O}\left(e^{\pi nt}\int_{C}^{\infty}\exp\left(-(n+1)u+2u\sum_{i=0}^{N}|p_{i}|\right)du\right)=\mathcal{O}\left(\frac{e^{\pi nt}}{n}e^{-nC}\right)
\]
for some large constant $C$. Recall that we are in the case $nt\ll\ln^{4}n$. Thus, if we choose $\epsilon$ so that besides satsifying the condition $n\epsilon=o(1)$,  we also have
\[
\frac{1}{\epsilon^{\sum_{i=1}^{N}|p_{i}|}}=\mathcal{O}\left(\exp\left(\frac{\ln^{5}n}{2}\right)\right),
\]
then 
\[
\int_{\ln^{5}n/n+i\epsilon}^{+\infty+i\epsilon}\psi(e^{z})e^{n\phi(e^{z},t)}e^{z}dz=\mathcal{O}\left(\exp\left(\pi nt-\frac{\ln^{5}n}{2}\right)\right).
\]
With a similar argument we also obtain the asymptotic expression
\[
\int_{+\infty-i\epsilon}^{\ln^{5}n/n-i\epsilon}\psi(e^{z})e^{n\phi(e^{z},t)}e^{z}dz=\mathcal{O}\left(\exp\left(\pi nt-\frac{\ln^{5}n}{2}\right)\right).
\]
From equations \eqref{eq:splitintegral} , \eqref{eq:mainint}, Lemma 10 in \cite{cft},
and the fact (see \cite[p.11]{cft}) that 
\[
2i\sin(\pi(c+p+1-int))\Gamma(c+p+1-int)
\]
is 
\[
\begin{cases}
\gg e^{-\ln^{4}n} & \text{if }e^{-\ln^{4}n}\ll nt=\mathcal{O}(1)\\
\gg\exp\text{\ensuremath{\left(n\pi t-\pi\ln^{4}n\right)}} & \text{if }1\ll nt=\mathcal{O}(\ln^{4}n)
\end{cases},
\]
we conclude that 
\[
\int_{+\infty}^{(0^{-})}\psi(e^{z})e^{n\phi(e^{z},t)}e^{z}dz\sim\frac{i\sin\pi(c+p+1-int)\Gamma(c+p+1-int)2^{c+p+1+int}\prod_{k=1}^{N}(1-\alpha_{i}^{2})^{p_{i}}}{n^{c+p+1-int}}
\]
uniformly on $e^{-\ln^{4}n}\ll nt=\mathcal{O}(\ln^{4}n)$,
completing the proof of the lemma.

\subsection{The case $\ln^{4}n/n\ll t<1$}
This section develops the asymptotic equivalence fromula for $p_n(t)$ on the remaining range of $t$.
\begin{lem}
\label{lem:globalasymp}Let $p_n(t)$ be as in \eqref{eq:pn}. Let $\epsilon=\ln n/\sqrt{n}$. Then as $n\rightarrow\infty$,
\[
p_{n}(t)\sim\psi(\zeta)e^{n\phi(\zeta,t)}\int_{-\epsilon}^{\epsilon}e^{-ny^{2}}z'(y)dy
\]
uniformly on $\ln^{4}n/n\ll t<1$. 
\end{lem}
\begin{proof}
We deform $\Gamma_{2}$ (see Figure \ref{fig:contour}) into the union of the curve $z(y)$, $-\infty<y<L$ and the vertical line $z(L)+iy$,
$0\le y<\infty$. (See Figure \ref{fig:deformed} for the new countour of integration.) On the curve $\Gamma_{\epsilon}=z(y),-\epsilon<y<\epsilon$
\begin{eqnarray}
\int_{\Gamma_{\epsilon}}e^{n\phi(z,t)}\psi(z)dz & = & \psi(\zeta)\int_{\Gamma_{\epsilon}}e^{n\phi(z,t)}dz\left(1+o(1)\right)\nonumber \\
 & = & \psi(\zeta)\int_{\Gamma_{\epsilon}}e^{n(\phi(z,t)-\phi(\zeta,t)+\phi(\zeta,t))}dz\left(1+o(1)\right)\nonumber \\
 & = & \psi(\zeta)e^{n\phi(\zeta,t)}\int_{-\epsilon}^{\epsilon}e^{-ny^{2}}z'(y)dy(1+o(1)).\label{eq:mainterm}
\end{eqnarray}

Next, we find an upper bound for the integral over the tail $z(y)$,
$-\infty<y<-\epsilon$. On the portion of this curve with a fixed
(independent of $n$) large finite length we have $\psi(z)=\mathcal{O}(1)$
and $\Re\phi(z,t)\le\phi(\zeta,t)-\epsilon^{2}$. Hence
the integral of $e^{n \phi(z,t)} \psi(z)$ on this portion is 
\[
\mathcal{O}\left(e^{n\phi(\zeta,t)-n\epsilon^{2}}\right).
\]
On the other hand, if $-\infty < y<-C$ for large $C>0$, then the equation
\[
-y^2=it \Log \frac{1+z}{1-z}-\Log z-\phi(\zeta,t)
\]
implies that $|z| \asymp e^{y^2}$. Thus from the equivalence 
\[
\phi_z(z,t)=-\frac{1}{z}+\frac{it}{1-z}+\frac{it}{1+z} \asymp \frac{1}{z}
\]
we deduce that 
\[
z'(y)=-\frac{2y}{\phi_z(z,t)}=\mathcal{O}(e^{2 y^2}).
\]
In addition, equation \eqref{eq:phidef} gives
\[
\psi(z(y))=\mathcal{O}(e^{Ay^2})
\]
for some $A>0$, from which we conclude that the integral over the tail $z(y), -\infty <y<-C$ is
\[
\mathcal{O} \left( e^{n\phi(\zeta,t)}\int_{-\infty}^{-C}e^{-(n-A-2)y^{2}}dy\right)=\mathcal{O}\left(e^{n\phi(\zeta,t)-(n-A-2)C^{2}}\right).
\]
Thus the integral over $z(y)$ for $-\infty<y<-\epsilon$ is $\mathcal{O} \left(e^{n \phi(\zeta,t)-n \epsilon^2} \right)$. Similarly, the integral over the tail $\epsilon<y<L$ is also $\mathcal{O}\left(e^{n\phi(\zeta,t)-n\epsilon^{2}}\right)$.

Next, we will show that integral over the vertical line $z(L)+iy$, $0\le y<\infty$
is also $\mathcal{O}\left(e^{n\phi(\zeta,t)-n\epsilon^{2}}\right)$.
We note that for any fixed $0<t<1$ and $z=z(L)+iy$
\begin{align*}
\frac{d\Re\phi(z,t)}{dy} & =\Re\left(i\phi_{z}(z,t)\right)\\
  &= \Re \left\{ i\left(\frac{it }{1+z}+\frac{it}{1-z}-\frac{1}{z} \right)\right\} \\
  & =\Re\left(-\frac{t}{1-z}-\frac{t}{1+z}-\frac{i}{z}\right).
\end{align*}
We substitute $z=z(L)+iy$ to the last expression and use a computer
algebra system to write this expression as 
\[
-t\frac{(1-z(L))|1-z|^2+(1+z(L))|1-z|^2}{|1-z^2|^2}-\frac{y}{z^2(L)+y^2}<0
\]
since $t \in (0,1), y \geq 0$ and $z(L)\in(0,1)$ by Lemma \ref{lem:zycurve}. Thus $\Re\phi(z(L)+iy,t)$
is a decreasing function in $y$. We conclude that for $z=z(L)+iy$,
\begin{equation}
\Re\phi(z,t)<\Re\phi(z(L),t)=\Re\phi(\zeta,t)-L^{2}.\label{eq:rephiineqvert}
\end{equation}
Moreover for $z=z(L)+iy$
\[
\psi(z)=z^{-1}\prod_{k=0}^{N}(1-\alpha_{i}z^{2})^{p_{i}}=\mathcal{O}(y^{A})
\]
for some $A>0$. We split the integral of interest as follows:
\[
\int_{z=z(L)+iy}e^{n\phi(z,t)}\psi(z)dz=i\int_{0}^{n}e^{n\phi(z,t)}\psi(z)dy+i\int_{n}^{\infty}e^{n\phi(z,t)}\psi(z)dy.
\]
By equation (\ref{eq:rephiineqvert}), the first integral on the right side
is 
\[
\mathcal{O}(n^{A+1}e^{n\Re\phi(\zeta,t)-nL^{2}})=\mathcal{O}\left(e^{n\phi(\zeta,t)-n\epsilon^{2}}\right).
\]
In the case $y>n$, we notice that
\[
e^{n\phi(z,t)}\psi(z)=\frac{(1-z)^{p+c-int}(1+z)^{p+c+int}\prod_{i=1}^{N}(1-\alpha_{i}^{2}z^{2})^{p_{i}}}{z^{n+1}}
\]
where 
\[
\left|(1-z)^{p+c-int}\right|=|1-z|^{p+c}e^{nt\Arg(1-z)}=\mathcal{O}(y^{p+c}e^{n\pi}).
\]
With the same bound for $\left|(1-z)^{p+c-int}\right|$, we conclude
that 
\[
\left|e^{n\phi(z,t)}\psi(z)\right|=\mathcal{O}\left(\frac{e^{2n\pi}}{y^{n-B}}\right)
\]
for some $B$, and consequently 
\[
\int_{n}^{\infty}e^{n\phi(z,t)}\psi(z)dy=\mathcal{O}\left(\frac{e^{2\pi n}}{n^{n-B}}\right),
\]
an expression which is $\mathcal{O}\left(e^{n\phi(\zeta,t)-n\epsilon^{2}}\right)$. 

It remains to show that the bound, $\mathcal{O}\left(e^{n\phi(\zeta,t)-n\epsilon^{2}}\right)$,
is little-Oh of (\ref{eq:mainterm}), thereby establishing the asymptotic dominance of the central piece of the integral. Clearly, 
\[
\left|\int_{-\epsilon}^{\epsilon}e^{-ny^{2}}z'(y)dy\right|>\left|\int_{-\epsilon}^{\epsilon}e^{-ny^{2}}\Im z'(y)dy\right|.
\]
Using the result that $\Im z'(y)>0$ for $y \in (-\infty, L)$ - see the Appendix for the details, -  the right side is at least 
\[
\int_{-\epsilon}^{0}e^{-ny^{2}}\Im z'(y)dy.
\]
By (\ref{eq:zderiv}) and Remark \ref{rem:Imphiznegy}, we see that if
$-\epsilon<y<0$, then 
\begin{align*}
\Im z'(y) & =\frac{|2y\Im\phi_{z}(z(y),t)|}{|\phi_{z}(z(y),t)|^{2}}\\
 & >\frac{2|y|}{|\phi_{z}(z(y),t)|^{2}}\left|\min\left\{-\frac{1}{2}\Im\left(\frac{1}{z}\right),-\frac{y^{2}}{4t(|z|^{2}-1)}\Im\left(\frac{1}{z}\right)\right\}\right|\\
 & \gg|y|^{3}\Im\frac{1}{z}.
\end{align*}
Since $\Im z(y)$ is increasing in $y$, for $-\epsilon <y<0$, 
\[
\Im\frac{1}{z}=-\frac{\Im z}{|z|^{2}}>-\frac{\Im\zeta}{|z|^{2}}=\frac{t}{|z|^{2}}\gg\frac{\ln^{4}n}{n}.
\]
Thus 
\begin{align*}
\int_{-\epsilon}^{0}e^{-ny^{2}}\Im z'(y)dy & \gg\frac{\ln^{4}n}{n}\int_{-\epsilon}^{0}e^{-ny^{2}}|y|^{3}dy\\
 & =-\frac{\ln^{4}n}{n}\left.\frac{e^{-ny^{2}}\left(ny^{2}+1\right)}{2n^{2}}\right|_{0}^{\epsilon}\\
 & \asymp\frac{\ln^{4}n}{n^{3}}.
\end{align*}
We conclude that the bound $\mathcal{O}\left(e^{n\phi(\zeta,t)-n\epsilon^{2}}\right)$ for the tail is little-Oh of (\ref{eq:mainterm}), since for $\epsilon=\ln n/\sqrt{n}$,
\[
e^{-n\epsilon^{2}}=e^{-\ln^{2}n}=o\left(\frac{\ln^{4}n}{n^{3}}\right).
\]
This finishes the proof of the lemma, and provides the claimed asymptotic equivalence for $p_n(t)$ uniformly on $\ln^n /n \ll t <1$.
\end{proof}

\subsection{The zero distribution} In this section we compute the change in the argument fo the asymptotic representation of $p_n(t)$, leading to the completion of the proof of the main result.
\begin{lem}
\label{lem:changeargglobal} Let $\phi(z,t), \psi(z)$ and $z(y)$ as in the preceding discussion. If 
\[
g(\zeta(t))=\psi(\zeta)e^{n\phi(\zeta,t)}\int_{-\epsilon}^{\epsilon}e^{-ny^{2}}z'(y)dy,
\]
then for any small $\tau$, 
\[
\Delta\arg_{\tau\le t<1}g(\zeta(t))=\frac{n\pi}{2}+n\tau\ln\tau-n(\ln2+1)\tau+\frac{\pi(1-p-c)}{2}+C+\mathcal{O}(n\tau^{2}+\tau)
\]
where $|C|\le\pi$. 
\[
.
\]
\end{lem}

\begin{proof}
The definition of $g(\zeta(t))$ yields 
\[
\Delta\arg_{\tau\le t<1}g(\zeta(t))=\Delta\arg_{\tau\le t<1}\psi(\zeta)+n\Delta\Im_{\tau\le t<1}\phi(\zeta,t)+\Delta\Arg_{\tau\le t<1}\int_{-\epsilon}^{\epsilon}e^{-ny^{2}}z'(y)dy.
\]
The estimate 
\[
\zeta(\tau)=1-i\tau-\frac{\tau^{2}}{2}+\mathcal{O}(\tau^{3})
\]
gives 
\begin{align*}
\Im\phi(\zeta(\tau),\tau) & =\tau\ln|1+\zeta(\tau)|-\tau\ln|1-\zeta(\tau)|-\Arg\zeta(\tau)\\
 & =\tau\ln2-\tau\ln\tau+\tau+\mathcal{O}(\tau^{2}).
\end{align*}
Utilizing the equality
\[
\lim_{t\rightarrow1}\Im\phi(\zeta,t)=\pi/2,
\]
we deduce that 
\[
n\Delta\arg_{\tau\le t<1}\Im\phi(\zeta,t)=\frac{n\pi}{2}+n\tau\ln\tau-n(\ln2+1)\tau+\mathcal{O}(n\tau^{2}).
\]

We now compute $\Delta\arg_{\tau<t<1}\psi(\zeta(t))$ where 
\[
\psi(\zeta)=\zeta^{-1}(1-\zeta)^{p+c}(1+\zeta)^{p+c}\prod_{k=1}^{N}(1-\alpha_{i}^{2}\zeta^{2})^{p_{i}} \qquad \textrm{(c.f. equation (\ref{eq:psidef})}.
\]
The representation 
\[
\psi(\zeta(\tau))=2^{p+c}(i\tau)^{p+c}\prod_{k=1}^{N}(1-\alpha_{i}^{2})^{p_{i}}(1+\mathcal{O}(\tau)),
\]
along with the limit
\[
\lim_{t\rightarrow1}\psi(\zeta(t))=i(1+i)^{p+c}(1-i)^{p+c}\prod_{k=1}^{N}(1+\alpha_{i}^{2})^{p_{i}}
\]
imply that 
\[
\Delta\arg_{\tau<t<1}\psi(\zeta(t))=\frac{\pi(1-p-c)}{2}+\mathcal{O}(\tau).
\]
Finally, since
\[
\Im\int_{-\epsilon}^{\epsilon}e^{-ny^{2}}z'(y)dy\ge0,
\]
it follows that
\[
\left| \Delta \arg_{\tau <t<1} \int_{-\epsilon}^{\epsilon}e^{-ny^{2}}z'(y)dy \right| \leq \pi.
\]
The proof is thus complete.
\end{proof}
\begin{lem}
\label{lem:changeargsmallt} Suppose $\tau_{1}$ and $\tau_{2}$ are
constant multiples of $e^{-\ln^{4}n}/n$ and $\ln^{4}n/n$ and $\alpha$
is the unique angle such that $-\pi<\alpha\le\pi$ and $(c+p+1/2)\pi=2k\pi+\alpha$
for $k\in\mathbb{Z}$. Then $p(\zeta(t))\ne0$ for $\tau_{1}\le t\le\tau_{2}$
and 
\[
\Delta\arg_{\tau_{1}\le t\le\tau_{2}}p(\zeta(t))=n\tau_{2}\ln2-n\tau_{2}\ln\tau_{2}+n\tau_{2}-\frac{|c+p|\pi}{2}-\frac{\pi}{4}-\eta+\mathcal{O}\left(\frac{1}{\ln^{4}n}\right)
\]
where 
\begin{equation}
\eta=\begin{cases}
-\pi/2 & \text{if }c+p<0\\
0 & \text{ if }c+p\ge0\text{ and }\alpha=\pm\frac{\pi}{2}\\
-\alpha & \text{ if }c+p\ge0\text{ and }-\pi/2<\alpha<\pi/2\\
-\alpha-\pi & \text{ if }c+p\ge0\text{ and }-\pi<\alpha<-\pi/2\\
-\alpha+\pi & \text{ if }c+p\ge0\text{ and }\pi/2<\alpha\le\pi.
\end{cases}.\label{eq:etadef}
\end{equation}
\end{lem}

\begin{proof}
Using Lemma 10 in \cite{cft} once more we see that as $n\rightarrow\infty$,
\[
p(\zeta(t))\sim\frac{i\pi2^{c+p+1+int}\prod_{k=1}^{N}(1-\alpha_{i}^{2})^{p_{i}}}{\Gamma(-c-p+int)n^{c+p+1-int}}
\]
uniformly on $\tau_{1}\le t\le\tau_{2}$. The change of argument for
$2^{c+p+1+int}$and $n^{c+p+1-int}$when $\tau_{1}\le t\le\tau_{2}$
are $n(\tau_{2}-\tau_{1})\ln2$ and $-n(\tau_{2}-\tau_{1})\ln n$
respectively. Following the computations in the proof
of Lemma 12 in \cite{cft}, we arrive at the equation 
\[
\Delta\arg_{\tau_{1}\le t\le\tau_{2}}\Gamma(-c-p+int)=n\tau_{2}\ln(n\tau_{2})-n\tau_{2}+\frac{|c+p|\pi}{2}+\frac{\pi}{4}+\eta+\mathcal{O}\left(\frac{1}{\ln^{4}n}\right)
\]
from which the result follows. 
\end{proof}
As a consequence of Lemmas \ref{lem:changeargglobal} and \ref{lem:changeargsmallt}
(where $\tau$ is a constant multiple of $\ln^{4}n/n$), we conclude
that 
\begin{equation}
\Delta\arg_{e^{-\ln^{4}n}/n\ll t<1}p(\zeta(t))=-\frac{|c+p|\pi}{2}-\frac{\pi}{4}-\eta+\frac{n\pi}{2}+\frac{\pi(1-p-c)}{2}+C+o(1).\label{eq:changearg}
\end{equation}
 
The completion of the proof of Theorem \ref{thm:pn_generalization}, i.e. the claim about the location of the zeros of the polynomials $p_n(t)$, is obtained using arguments completely analogous to those in \cite{cft}, after comparing the right hand side of equation \eqref{eq:changearg} to the absolute value of (3.21) in \cite{cft}.

\subsection{The limiting zero distribution density function}

In this section, we compute the limiting distribution density function $D(x)$
of the zeros of $h_{n}(c-int)$ on $t\in(0,1)$. This function at
$x\in(0,1)$ is given by 
\[
\lim_{\epsilon\rightarrow0}\frac{1}{\epsilon}\lim_{n\rightarrow\infty}\frac{N_{n,\epsilon}(x)}{n}
\]
where $N_{n,\epsilon}(x)$ denote the number of zeros of $h_{n}(c-int)$
on the interval $t\in(x,x+\epsilon)$. We recall from Lemma \ref{lem:globalasymp}
that for any $x\in(0,1)$ and small $\epsilon$ 
\[
p_{n}(\zeta(t))\sim\psi(-i)e^{n\phi(\zeta,t)}\int_{-\epsilon}^{\epsilon}e^{-ny^{2}}z'(y)dy
\]
uniformly on $t\in(x,x+\epsilon)$. Since 
\[
\Im\int_{-\epsilon}^{\epsilon}e^{-ny^{2}}z'(y)dy>0,
\]
we have 
\[
\Delta\arg_{x<t<x+\epsilon}p_{n}(\zeta(t))=n\Delta\Im_{x<t<x+\epsilon}\phi(\zeta,t)+C
\]
for some $|C|<\pi$. It is immediate from the Taylor expansion of
$\phi(\zeta,\cdot)$ about $x$ that 
\[
\left.\phi(\zeta,t)\right|_{t=x}^{t=x+\epsilon}=\left.\frac{d\phi(\zeta,t)}{dt}\right|_{t=x}\epsilon+\mathcal{O}(\epsilon^{2}),
\]
where $\phi(z,t)=it\Log(1+z)-it\Log(1-z)-\Log z$.
\begin{align*}
\left.\frac{d\phi}{dt}\right|_{t=x} & =\left.\frac{\partial\phi}{\partial\zeta}\right|_{t=x}\left.\frac{d\zeta}{dt}\right|_{t=x}+\left.\frac{\partial\phi}{\partial t}\right|_{t=x}\\
 & =\left.\frac{\partial\phi}{\partial t}\right|_{t=x}\\
 & =i\left(\Log(1+\zeta(x))-\Log(1-\zeta(x))\right).
\end{align*}
Since $\pi h_{n}(c-int)$ is the imaginary part, or $-i$ times the
real part of 
\[
p_{n}(t)=\int_{\Gamma_{2}}\psi(z)e^{n\phi(z,t)}dz,
\]
we conclude that 
\begin{align*}
\lim_{\epsilon\rightarrow0}\frac{1}{\epsilon}\lim_{n\rightarrow\infty}\frac{N_{n,\epsilon}(x)}{n} & =\lim_{\epsilon\rightarrow0}\frac{1}{\epsilon}\lim_{n\rightarrow\infty}\frac{\left|\Delta\arg_{x<t<x+\epsilon}p_{n}(\zeta(t))\right|}{\pi n}\\
 & =\frac{1}{\pi}\ln\left|\frac{1+\zeta(x)}{1-\zeta(x)}\right|.
\end{align*}
Substituting $\zeta(x)=-ix+\sqrt{1-x^{2}}$, we obtain the following limiting
distribution density function (see Figure \ref{fig:limitdensity})
\[
D(x)=\frac{1}{2\pi}\ln\frac{1+\sqrt{1-x^{2}}}{1-\sqrt{1-x^{2}}}.
\]

\begin{figure}

\begin{centering}
\includegraphics[width=3 in]{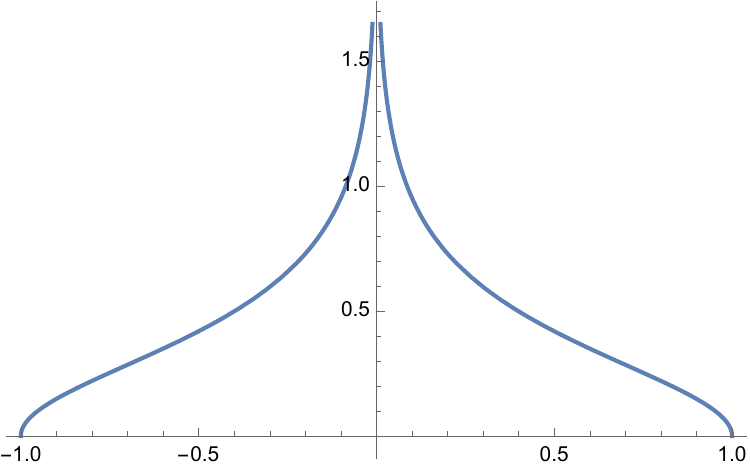}
\par\end{centering}
\caption{\label{fig:limitdensity}The limiting zero distribution density function}

\end{figure}
\section*{Appendix}
The work in this Appendix is dedicated to showing that for any $y\in(-\infty,L)$, $\Im z'(y)>0$. This result is used in establishing asymptotic estimates of the tails of the integral representations of the $p_n(t)$. As such, it is essential to the main theorem, but the tools used to prove the result are elementary and somewhat tedious.
We begin by differentiating both sides of (\ref{eq:zyfunc}) with respect to
$y$ and obtain 
\begin{equation}
z'(y)=-\frac{2y}{\phi_{z}(z(y),t)}=-\frac{2 y(\Re \phi_z(z(y),t)-i \Im \phi_z (z(y),t))}{|\phi_z(z(y),t)|^2}.\label{eq:zderiv}
\end{equation}
The fact that $\Im z'(y)>0$ follows from the fact that $y$ and $\Im \phi_z(z(y),t)$ carry the same sign when $y \neq 0$. The claim is still true when $y=0$, although we don't really need it, since the estimates we use are all inside integrals.  \\
We first treat the case when $y>0$.
\begin{lem}
For any $y\in(0,L)$ and $t\in(0,1)$, $\Im\phi_{z}(z(y),t)>0$. 
\end{lem}

\begin{proof}
We note that for small $y$ 
\[
\phi_{z}(z(y),t)=\phi_{z^{2}}(\zeta,t)(z(y)-\zeta)+\mathcal{O}(z(y)-\zeta)^{2}.
\]
We deduce from Lemma \ref{lem:phi2zeta} and equation \eqref{eq:zsmally} that
$\Im\phi_{z}(z(y),t)>0$ for small $y$. Suppose now by way of contradiction, that
$\Im\phi_{z}(z(y),t)>0$ does not hold for all $y>0$. Then by the
Intermediate Value Theorem, there exists $y_{0}>0$ such that $\Im\phi_{z}(z(y_{0}),t)=0$.
We invert the relation 
\[
y=\phi_{z}(z,t)-\Re\phi_{z}(z(y_{0}),t)
\]
to obtain a function $\gamma(y)$ in a neighborhood of $0$ such that 
\begin{equation} \label{eq:ygamma}
y=\phi_{z}(\gamma(y),t)-\Re\phi_{z}(z(y_{0}),t)
\end{equation}
and $\gamma(0)=z(y_{0})$. We apply analytic continuation arguments
similar to those in Lemma \ref{lem:zycurve} to obtain a curve $\gamma(y)$,
$0<y<\tilde L$ for $0< \tilde L\le\infty$, inside the open
unit ball in the fourth quadrant.
\begin{figure}[htbp]
\begin{center}
\includegraphics[width=3 in]{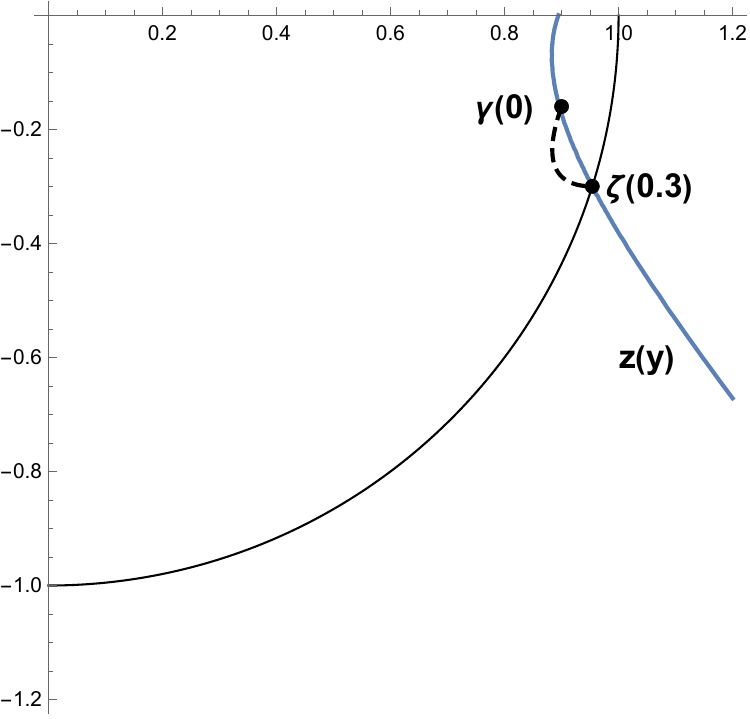}
\caption{The three essential curves in our proof: the $\zeta$ curve, the $z(y)$ curve and the $\gamma(y)$ curve}
\label{fig:gammacurve}
\end{center}
\end{figure}

We note that $\gamma(\tilde L):=\lim_{y\rightarrow \tilde L}\gamma(y)$ must lie
on the boundary of this region which is the union of $[0,1]$, $i[-1,0]$,
and the circular arc $e^{i\theta}$, $-\pi/2<\theta<0$. \\
We claim that
$\gamma(\tilde L)=\zeta$. (See the dashed line in Figure \ref{fig:gammacurve}) First note that $\gamma(\tilde L)\notin(0,1)$ since
if $\gamma(\tilde L)\in(0,1)$, then on the one hand $\Im \phi_z(\gamma(\tilde L),t)=0$, since $\phi_z$ is real along the curve $\gamma(y)$. On the other hand,
\[
\Im\phi_{z}(\gamma(\tilde L),t)=\Im\left(\frac{it}{1+\gamma(\tilde L)}+\frac{it}{1-\gamma(\tilde L)}-\frac{1}{\gamma(\tilde L)}\right)\ne0.
\]
Similarly $\gamma(\tilde L)\notin i[-1,0)$ since if $\gamma(\tilde L)=iv$ for
$v\in[-1,0)$, then 
\begin{align*}
0=\Im\phi_{z}(\gamma(\tilde L),t) & =\Im\left(\frac{it}{1+iv}+\frac{it}{1-iv}-\frac{1}{iv}\right)\\
 & =\frac{2t}{1+v^{2}}+\frac{1}{v}\\
 & =\frac{1+v^{2}+2tv}{v(1+v^{2})}<0.
\end{align*}
We also have $\gamma(\tilde L)\ne0$ since if $\gamma(\tilde L)=0$, then taking
the limit of both sides of equation \eqref{eq:ygamma} as $y\rightarrow \tilde L$ to obtain 
\[
\tilde L=2it-\lim_{y\rightarrow \tilde L}\frac{1}{\gamma(y)}-\Re\phi_{z}(z(y_{0}),t).
\]
Since the right side approaches $\infty$, we conclude that $L=+\infty$.
Taking the real part of both sides above we obtain a contradiction
since 
\[
\Re\frac{1}{\gamma(y)}=\frac{\Re\gamma(y)}{|\gamma(y)|^{2}}>0.
\]
It follows that $\gamma(\tilde L)=u+iv$ lies on the circular arc $z=e^{i \theta}$, $-\pi/2 <\theta<0$. One readily computes
\[
0=\Im\phi_{z}(\gamma(\tilde L),t)=t+v.
\]
Hence $v=-t$ and $\Im (\zeta(\tilde L))=\Im \zeta$, and hence $\gamma(\tilde L)=\zeta$. This establishes the claim.\\
We continue by noting that
\[
\Im\phi(\gamma(0),t)=\Im\phi(z(y_{0}),t)=\Im\phi(\zeta,t)
\]
by the definition of the $z(y)$ curve, and 
\[
\Im\phi(\gamma(L),t)=\Im\phi(\zeta,t).
\]
An application of the Mean Value Theorem, yields $y_{1}\in(0,L)$ such that 
\begin{equation}
\frac{d}{dy} \Im \phi(\gamma(y),t) \Big|_{y=y_1}=\Im\left(\phi_{z}(\gamma(y_{1}),t)\gamma'(y_{1})\right)=\Im \phi_z(\gamma(y_{1}),t) \Re \gamma'(y_1)+\Re \phi_z(\gamma(y_{1}),t) \Im \gamma'(y_1)=0.\label{eq:Imphigamma}
\end{equation}
Recall that $\Im\phi_{z}(\gamma(y_{1}),t)=0$ and $\Re\phi_{z}(\gamma(y_{1}),t)\ne0$
(as $\zeta$ is the only solution of $\phi_{z}(z,t)=0$ on the fourth
quadrant). Therefore equation \eqref{eq:Imphigamma} implies that 
\[
\Im\gamma'(y_{1})=0.
\]
We now differentiate both sides of equation \eqref{eq:ygamma}
\[
y=\phi_{z}(\gamma(y),t)-\Re\phi_{z}(z(y_{0}),t)
\]
with respect $y$ to obtain 
\begin{equation}
\gamma'(y)=\frac{1}{\phi_{z^{2}}(\gamma(y),t)},\label{eq:gammaderiv}
\end{equation}
and consequently 
\[
\Im\phi_{z^{2}}(\gamma(y_{1}),t)=0.
\]
Write $\gamma(y_{1})=u+iv$. Using a CAS we find that the equations
$\Im\phi_{z}(\gamma(y_{1}),t)=0$ and $\Im\phi_{z^{2}}(\gamma(y_{1}),t)=0$
give 
\[
2tu^{4}-2tu^{2}-2tv^{4}-2tv^{2}-u^{4}v-2u^{2}v^{3}+2u^{2}v-v^{5}-2v^{3}-v=0
\]
and 
\begin{align*}
2tu^{4}-4tu^{6}+2tu^{8}-v+4u^{2}v-6u^{4}v+4u^{6}v-u^{8}v+4tu^{2}v^{2}-12tu^{4}v^{2}-4v^{3}+4u^{2}v^{3}+4u^{4}v^{3}\\
-4u^{6}v^{3}+2tv^{4}-12tu^{2}v^{4}-12tu^{4}v^{4}-6v^{5}-4u^{2}v^{5}-6u^{4}v^{5}-4tv^{6}-16tu^{2}v^{6}-4v^{7}-4u^{2}v^{7}-6tv^{8}-v^{9} & =0.
\end{align*}
We solve the first equation for $t$: 
\begin{equation}
t=\frac{u^{4}v+2u^{2}v^{3}-2u^{2}v+v^{5}+2v^{3}+v}{2\left(u^{4}-u^{2}-v^{4}-v^{2}\right)}\label{eq:tvalue}
\end{equation}
Substituting into the second equation we obtain
\[
2v\left(u^{2}+v^{2}\right)\left(u^{2}-2u+v^{2}+1\right)\left(u^{2}+2u+v^{2}+1\right)\left(2u^{4}v^{2}-u^{4}+4u^{2}v^{4}+6u^{2}v^{2}+2u^{2}+2v^{6}-v^{4}-4v^{2}-1\right)=0,
\]
or equivalently 
\begin{equation}
2u^{4}v^{2}-u^{4}+4u^{2}v^{4}+6u^{2}v^{2}+2u^{2}+2v^{6}-v^{4}-4v^{2}-1=0.\label{eq:Imphi2}
\end{equation}
We apply the Lagrange multiplier method to find the minimum value of 
\[
t=\frac{u^{4}v+2u^{2}v^{3}-2u^{2}v+v^{5}+2v^{3}+v}{2\left(u^{4}-u^{2}-v^{4}-v^{2}\right)}
\]
under the conditions (\ref{eq:Imphi2}), $u^{2}+v^{2}\le1$, and $0\le t\le1$
and find that the numerical value of this minimum value is $0.924256...$
We now look for an additional constraint on the value of $t$. To this end, note that 
\[
\Re\phi(\gamma(0),t)=\Re\phi(z(y_{0}),t)<\Re\phi(\zeta,t)
\]
by the definition of $z(y)$ curve. Moreover, as $y\rightarrow \tilde L$
\[
\Re\phi(\gamma(y),t)\sim\Re\left(\phi(\zeta,t)+\frac{1}{2}\phi_{z^{2}}(\zeta,t)\gamma'(\zeta)^{2}(y-\tilde L)^{2}\right).
\]
We conclude from (\ref{eq:gammaderiv}) that for $y$ sufficiently
close to $\tilde L$, $\Re\phi(\gamma(y),t)>\Re\phi(\zeta,t)$. By the Intermediate
Value Theorem, there is $y^{*}\in(0,\tilde L)$ such that $\Re\phi(\gamma(y^{*}),t)=\Re\phi(\zeta,t)$.
Together with $\Re\phi(\gamma(\tilde L),t)=\Re\phi(\zeta,t)$, we apply Mean
Value Theorem to conclude that there is $y_{1}^{*}\in(y^{*},L)$ such that
\[
\Re\left(\phi_{z}(\gamma(y_{1}^{*}),t)\gamma'(y_{1}^{*})\right)=0.
\]
We apply the same arguments as those in (\ref{eq:Imphigamma}) to
conclude $\Im\phi_{z}(\gamma(y_{1}^{*}),t)=0$ and $\Re\phi_{z^{2}}(\gamma(y_{1}^{*}),t)=0$.
Letting $\gamma(y_{1}^{*})=u+iv$, we substitute (\ref{eq:tvalue})
to $\Re\phi_{z^{2}}(\gamma(y_{1}^{*}),t)=0$ and obtain 
\[
-u^{8}-6u^{6}v^{2}+3u^{6}-8u^{4}v^{4}-u^{4}v^{2}-3u^{4}-2u^{2}v^{6}+13u^{2}v^{4}+8u^{2}v^{2}+u^{2}+v^{6}-v^{4}-v^{2}+v=0.
\]
Using the Lagrange multiplier method once more, we find the numerical maximum value of
$t$ under this condition and $u^{2}+v^{2}\le1$ and $0\le t\le1$
to be $0.707107...$. Thus there are no values of $t$ under which the conditions implied by our assumptions hold. 
\end{proof}
Next we treat the case of negative $y$.
\begin{lem}
\label{lem:negyImphi_z}For any $y\in(-\infty,0)$ and $t\in(0,1)$,
$\Im\phi_{z}(z(y),t)<0$. 
\end{lem}

\begin{proof}
Let $y\in(-\infty,0)$ and $t\in(0,1)$ be given. We use equation
(\ref{eq:zyfunc}) to deduce that for $z:=z(y)$, 
\begin{equation}
\Re\phi(z,t)-\Re\phi(\zeta,t)<0,\label{eq:Rephiineq}
\end{equation}
or equivalently, 
\begin{equation} \label{eq:lem25}
t\left(\frac{\pi}{2}+\Arg(1+z)-\Arg(1-z)\right)+\ln|z|>0.
\end{equation}
Since $y<0$, we see that $|z(y)|>1$ and 
\[
|1-z|^{2}+|1+z|^{2}=2+2|z|^{2}<4|z|^{2}.
\]

\begin{figure}

\begin{centering}
\includegraphics[scale=0.3]{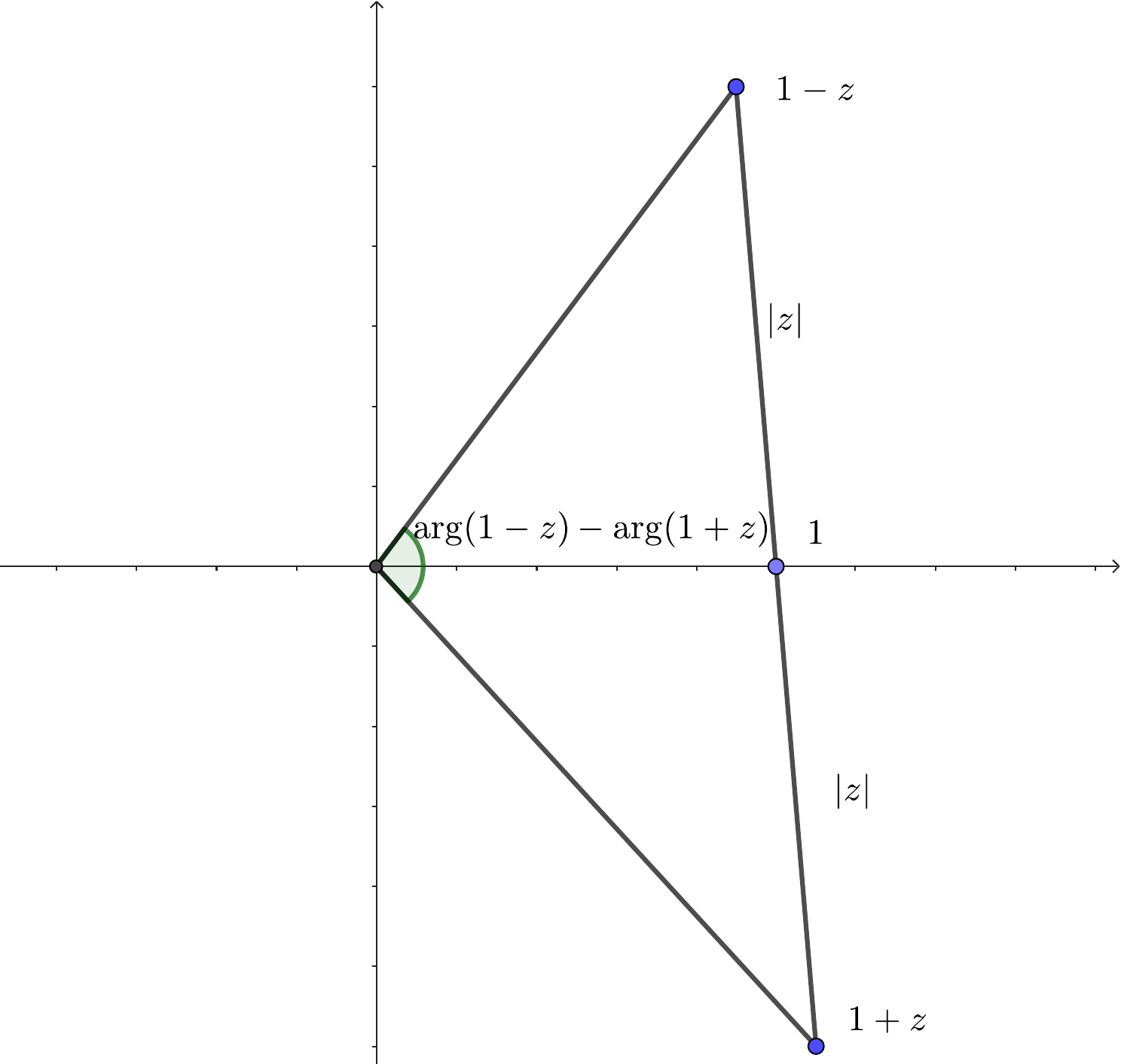}
\par\end{centering}
\caption{\label{fig:triangle}The triangle with vertices $0$, $1-z$, and
$1+z$}

\end{figure}

As Figure \ref{fig:triangle} illustrates, $|1-z|$, $|1+z|$, and
$2|z|$ are the lengths of the three sides of the triangle with vertices
$0,$$1-z$, and $1+z$, and $-\Arg(1+z)+\Arg(1-z)$ is the angle opposite
the the side of length $2|z|$. The above inequality therefore implies
that the angle opposite the side with length $2|z|$ is obtuse: 
\[
\pi>-\Arg(1+z)+\Arg(1-z)>\pi/2.
\]
Combining with equation \eqref{eq:lem25} we get
\begin{align*}
t & <\frac{-\ln|z|}{\frac{\pi}{2}+\Arg(1+z)-\Arg(1-z)}\\
 & <\frac{-\ln|z|}{\sin(\frac{\pi}{2}+\Arg(1+z)-\Arg(1-z))}.
\end{align*}
Using the identities $\sin(\pi/2+x-y)=\cos(x-y)=\cos x\cos y+\sin x\sin y$
we find that 
\begin{align*}
\sin\left(\frac{\pi}{2}+\Arg(1+z)-\Arg(1-z)\right) & =\cos\left(\Arg(1+z)\right)\cos\left(\Arg(1-z)\right)+\sin\left(\Arg(1+z)\right)\sin\left(\Arg(1-z)\right)\\
 & =\frac{1+\Re z}{|1+z|}\cdot\frac{1-\Re z}{|1-z|}-\frac{\Im z}{|1+z|}\cdot\frac{\Im z}{|1-z|}\\
 & =\frac{1-|z|^{2}}{|1-z||1+z|}
\end{align*}
which in turn yields 
\begin{equation}
t<\frac{|z-1||z+1|\ln|z|}{|z|^{2}-1}.\label{eq:tineq}
\end{equation}
With this upper bound on $t$ we calculate a bound on $\Im\phi_{z}(z,t)$. Recall from equation \eqref{eq:phiz} that
\begin{align}
\Im\phi_{z}(z,t) & =t\Re\left(\frac{2}{1-z^{2}}\right)-\Im\left(\frac{1}{z}\right)\nonumber \\
 & =2t\frac{\Re(1-z^{2})}{|z^{2}-1|^2}-\Im\left(\frac{1}{z}\right).\label{eq:Imphizineq}
\end{align}
If $\Re(1-z^{2})\le0$, then $\Im\phi_{z}(z,t)<0$ since $\Im(1/z)>0$
as $z$ lies in the open fourth quadrant. On the other hand if $\Re(1-z^{2})>0$,
then equation \eqref{eq:tineq} implies that
\[
\Im\phi_{z}(z,t)<\frac{2\ln|z|\Re(1-z^{2})}{(|z|^{2}-1)|z^{2}-1|}-\Im\left(\frac{1}{z}\right).
\]
We let 
\[
f(z)=\frac{2\ln|z|\Re(1-z^{2})}{(|z|^{2}-1)|z^{2}-1|}-\Im\left(\frac{1}{z}\right)
\]
with the convention that 
\[
\frac{2\ln|z|}{|z|^{2}-1}=1
\]
for $|z|=1$. It remains to show that $f(z)<0$ for all $z$ in the fourth
quadrant satisfying $|z|>1$ and $\Re(1-z^{2})>0$. To this end, write
$z=re^{i\theta}$, $-\pi/2<\theta<0$, and $r>1$. Then
\begin{equation}
f(z)=f(r,\theta)=\frac{2(\ln r)(1-r^{2}\cos2\theta)}{(r^{2}-1)\sqrt{r^{4}+1-2r^{2}\cos2\theta}}+\frac{\sin\theta}{r}.\label{eq:fzpolar}
\end{equation}
The condition $\Re(1-z^{2})>0$ gives $\cos2\theta=1-2\sin^{2}\theta<1/r^{2}$
or equivalently 
\[
\sin\theta<-\sqrt{\frac{1}{2}-\frac{1}{2r^{2}}}.
\]
We combine the two fractions on the right side of (\ref{eq:fzpolar})
and conclude that $f(z)<0$ if and only if 
\[
2(\ln r)(1-r^{2}\cos2\theta)r<-\sin\theta(r^{2}-1)\sqrt{r^{4}+1-2r^{2}\cos2\theta}.
\]
Since both sides of the inequality above are positive, we compare
the squares of both sides and it suffices to show
\[
g(r,s):=4\ln^{2}r(1-r^{2}(1-2s))^{2}r^{2}-s(r^{2}-1)(r^{4}+1-2r^{2}(1-2s))<0
\]
where $r>1$ and 
\[
\frac{1}{2}-\frac{1}{2r^{2}}<s:=\sin^{2}\theta<1.
\]
Since 
\[
\frac{\partial g}{\partial s}=-r^{8}-8r^{6}s+32r^{6}s\ln^{2}r+4r^{6}-16r^{6}\ln^{2}r+16r^{4}s-6r^{4}+16r^{4}\ln^{2}r-8r^{2}s+4r^{2}-1
\]
is linear in $s$, $g(r,s)$ has a most one critical value as a function
in $s$ on $[1/2-1/(2r^{2}),1]$. One easily checks that 
\[
\left.\frac{\partial g}{\partial s}\right|_{(r,1/2-1/(2r^{2}))}=-(r-1)^{3}(r+1)^{3}\left(r^{2}+3\right)<0
\]
and consequently $g(r,s)$ cannot attain its maximum value at that
potential critical point (if such exists). We conclude that for each
$r$, the maximum value of $g(r,s)$ occur at either $g(r,1)$ or
$g(r,1/2-1/(2r^{2})]$. Finally, since for any $r>1$
\begin{align}
g(r,1) & =-\left(r^{2}+1\right)^{2}\left(r^{2}-2r\log(r)-1\right)\left(r^{2}+2r\log(r)-1\right)<0\label{eq:grright}\\
g(r,1/2-1/(2r^{2})) & =-\frac{(r-1)^{4}(r+1)^{4}\left(r^{2}+1\right)}{2r^{2}}<0\label{eq:grleft}
\end{align}
we conclude that $g(r,s)<0$ for all $r>1$ and $s\in[1/2-1/(2r^{2})]$.
Thus, for all $y\in(-\infty,0)$ and $t\in(0,1)$, $\Im\phi_{z}(z(y),t)<0$,
as desired. 
\end{proof}
\begin{rem}
\label{rem:Imphiznegy}In the case $y<0$, we can slightly improve
the conclusion of Lemma \ref{lem:negyImphi_z} by the conclusion that
\[
\Im\phi_{z}(z,t)\le\min\left(-\frac{1}{2}\Im\left(\frac{1}{z}\right),-\frac{y^{2}}{4t(|z|^{2}-1)}\Im\left(\frac{1}{z}\right)\right)<0.
\]
Indeed, if we replace (\ref{eq:Rephiineq}) 
\[
\Re\phi(z,t)-\Re\phi(\zeta,t)=-y^{2},
\]
then (\ref{eq:tineq}) becomes 
\[
t<\frac{|z-1||z+1|(\ln|z|-y^{2})}{|z|^{2}-1}
\]
from which we have $\ln|z|>y^{2}$. Now, if 
\[
\Re(1-z^{2})\le\frac{|z^{2}-1|}{4t}\Im\left(\frac{1}{z}\right),
\]
then from (\ref{eq:Imphizineq}), we have 
\[
\Im\phi_{z}(z,t)\le-\frac{1}{2}\Im\left(\frac{1}{z}\right).
\]
On the other hand, if 
\[
\Re(1-z^{2})>\frac{|z^{2}-1|}{4t}\Im\left(\frac{1}{z}\right)>0
\]
then 
\begin{align*}
\Im\phi_{z}(z,t) & <\frac{2(\ln|z|-y^{2})\Re(1-z^{2})}{(|z|^{2}-1)|z^{2}-1|}-\Im\left(\frac{1}{z}\right)\\
 & <f(z)-\frac{y^{2}}{4t(|z|^{2}-1)}\Im\left(\frac{1}{z}\right)\\
 & <-\frac{y^{2}}{4t(|z|^{2}-1)}\Im\left(\frac{1}{z}\right).
\end{align*}
\end{rem}


\end{document}